\documentclass[11pt,reqno]{amsart}
\usepackage{amsmath, amssymb, amsthm}
\usepackage[linktocpage=true]{hyperref}
\hypersetup{colorlinks,linkcolor=blue,urlcolor=blue,citecolor=blue}
\usepackage{mathtools} 
\usepackage[table]{xcolor}
\usepackage{stmaryrd} 
\usepackage{bbm} 
\usepackage{leftidx} 
\usepackage{mathabx} 
\usepackage{rotating} 
\usepackage{wasysym} 
\usepackage[mmddyy]{datetime}
\usepackage{comment}
\usepackage{makecell}
\usepackage[bbgreekl]{mathbbol} 
\DeclareSymbolFontAlphabet{\mathbbm}{bbold}
\DeclareSymbolFontAlphabet{\mathbb}{AMSb}
\usepackage{arydshln}
\makeatletter
\renewcommand*\env@matrix[1][*\c@MaxMatrixCols c]{%
  \hskip -\arraycolsep
  \let\@ifnextchar\new@ifnextchar
  \array{#1}}
\makeatother
\usepackage{youngtab,young}

\makeatletter
\newcommand{\labeltarget}[1]{\Hy@raisedlink{\hypertarget{#1}{}}}
\makeatother
\usepackage{todonotes}
\usepackage{pgf, tikz, colortbl}
\usetikzlibrary{cd}

\numberwithin{equation}{subsection}
\usepackage[margin = 0.9in]{geometry}
\DeclareMathSymbol{\hp}{\mathord}{AMSa}{"39}

\newcommand{\CC}{\mathbb{C}}

\newcommand{\QQ}{\mathbb{Q}}
\newcommand{\ZZ}{\mathbb{Z}}

\newcommand{\mrm}{\mathrm}

\usepackage[shortlabels]{enumitem}

\newenvironment{enua}{\begin{enumerate}[label=\textup{(\alph*)}]
}{\end{enumerate}}

\usepackage[capitalise]{cleveref}
\theoremstyle{plain}
\newtheorem{thm}{Theorem}[subsection]
\newtheorem{prop}[thm]{Proposition}
\newtheorem{lem}[thm]{Lemma}
\newtheorem{cor}[thm]{Corollary}

\newtheorem{mainthm}{Theorem}

\newtheorem{maincor}[mainthm]{Corollary}

\theoremstyle{definition}
\newtheorem{defn}[thm]{Definition}

\theoremstyle{remark}
\newtheorem{rmk}[thm]{Remark}
\newtheorem{expl}[thm]{Example}

\AddToHook{env/prop/begin}{\crefalias{thm}{prop}}
\AddToHook{env/lem/begin}{\crefalias{thm}{lem}}
\AddToHook{env/cor/begin}{\crefalias{thm}{cor}}
\AddToHook{env/conj/begin}{\crefalias{thm}{conj}}
\AddToHook{env/defn/begin}{\crefalias{thm}{defn}}
\AddToHook{env/notation/begin}{\crefalias{thm}{notation}}
\AddToHook{env/rmk/begin}{\crefalias{thm}{rmk}}
\AddToHook{env/expl/begin}{\crefalias{thm}{expl}}

\Crefname{thm}{Theorem}{Theorems}
\Crefname{lem}{Lemma}{Lemmas}
\Crefname{prop}{Proposition}{Propositions}
\Crefname{cor}{Corollary}{Corollaries}
\Crefname{conj}{Conjecture}{Conjectures}
\Crefname{defn}{Definition}{Definitions}
\Crefname{notation}{Notation}{Notations}
\Crefname{rmk}{Remark}{Remarks}
\Crefname{expl}{Example}{Examples}

\newcommand{\define}[4]{\expandafter#1\csname#3#4\endcsname{#2{#4}}}
\forcsvlist{\define{\DeclareMathOperator}{}{}}{ad,Ad,codim,GL,SL,vdim} 
\forcsvlist{\define{\newcommand}{\mathbf}{b}}{A,C,D,H,S,T,v,w,Z,a,b,i,j,k,1,K} 
\forcsvlist{\define{\newcommand}{\mathbb}{bb}}{A,B,D,H,N,Z,F,Q,R,P,S,C,T,L,G,V} 
\forcsvlist{\define{\newcommand}{\mathcal}{c}}{A,B,C,D,E,F,G,H,I,J,K,L,M,N,O,P,Q,R,S,T,U,V,W,X,Y,Z} 
\forcsvlist{\define{\newcommand}{\mathfrak}{f}}{a,h,g,gl,H,I,p,S,sl,z} 
\forcsvlist{\define{\newcommand}{\mathrm}{}}{crit,id,Id,pt,op,sm,gdim,aff,Ann,Aut,asph,colim,End,ext,Ext,Frac,gr,Gr,Hom,Ind,Ker,Mat,mon,Proj,Res,RHom,rk,Span,sph,Stab,Sym,sgn,Tor,Tr,triv,ind} 
\forcsvlist{\define{\DeclareMathOperator}{\mathsf}{}}{BM,Bun,Coh,Coker,Cone,Fl,Fun,Graph,Hecke,Higgs,loc,Pol,Quot,Rep,Spec,Supp,Vect} 
\renewcommand{\mod}{{\mathrm{mod}}}
\def\<{\langle}%
\def\>{\rangle}%
\def\-{\text{-}}%
\newcommand{\al}{\gamma}
\newcommand{\alb}{\overline{\gamma}}

\newcommand{\Ip}{I^{(1)}}

\newcommand{\HGI}{\cH_I} 
\newcommand{\HGIp}{\cH^{(1)}} 
\newcommand{\HGbI}{\widetilde{\cH}_I} 
\newcommand{\HGbIp}{\widetilde{\cH}^{(1)}} 
\newcommand{\HGbIpk}{\widetilde{\cH}^{(1)}_\kappa} 
\newcommand{\HKbI}{\widetilde{\cH}_\kappa} 
\newcommand{\HKbIp}{\widetilde{\cH}^{(1)}_\kappa} 
\newcommand{\Vk}{\widetilde{\cW}_{\kappa}} 
\newcommand{\WIp}{{\cW}^{\Ip}} 
\newcommand{\WbIp}{\widetilde{\cW}^{\Ip}} 
\newcommand{\WbI}{\widetilde{\cW}^{I}} 
\newcommand{\SGI}{\cS} 
\newcommand{\SGIp}{\cS^{(1)}} 
\newcommand{\SGbI}{\widetilde{\cS}} 
\newcommand{\SGbIp}{\widetilde{\cS}^{(1)}} 
\newcommand{\VGGb}{\widetilde{\mathcal{W}}} 
\newcommand{\HY}{\cH_Y} 
\newcommand{\HAY}{\cH_\aff^Y} 
\newcommand{\HA}{\cH_\aff} 

\newcommand{\tcH}{\bbH} 
\newcommand{\Fr}{\bbF} 
\newcommand{\Waff}{W_{\operatorname{aff}}} 
\newcommand{\Wp}{W^{(1)}} 
\newcommand{\sroots}{\Delta} 
\newcommand{\sroot}{\alpha} 
\newcommand{\Gk}{G_{\kappa}} 
\newcommand{\Bk}{B_{\kappa}} 
\newcommand{\Uk}{U_{\kappa}} 
\newcommand{\Tk}{T_{\kappa}} 
\renewcommand{\GL}{\mathbf{GL}} 

\newcommand{\redp}{\mathfrak{r}} 
\newcommand{\mG}{\widetilde{G}} 
\newcommand{\mK}{\widetilde{K}} 
\newcommand{\Bs}{\mathsf{B}} 
\newcommand{\Qs}{\mathsf{Q}} 
\newcommand{\mY}{Y_{(\mathsf{Q},n)}} 
\newcommand{\barn}{\bar{n}} 
\newcommand{\bg}{\mathbf{g}} 
\newcommand{\bq}{\mathbf{q}} 

\newcommand{\tp}[1]{{}_{#1}t'}
\begin{document}

\title{Quantum wreath products and $p$-adic general linear group}

\begin{abstract}
We study the pro-$p$ Iwahori--Hecke algebra and its Gelfand--Graev modules for the $p$-adic general linear group and its metaplectic covers. 
We develop the theory of quantum wreath products of skew polynomial type and use it to provide transparent descriptions of these Hecke algebras and their modules that were previously inaccessible through standard $p$-adic methods. 
We introduce the notion of (anti)spherical and Kashiwara--Miwa--Stern modules for these quantum wreath products for the first time and interpret the $p$-adic Gelfand--Graev modules in terms of these new modules. 

As an application, we study the structure theory for the corresponding $p$-adic pro-$p$ Schur algebras and obtain an explicit basis and multiplication rules. 
Moreover, we give algebraic (re)proofs of several results of $p$-adic interest including the existence of PBW basis for the pro-$p$ metaplectic and Iwahori--Hecke algebras, identification of the Iwahori--Schur algebra with the quantum affine Schur algebra,
and failure of the local Shimura correspondence at the pro-$p$ level. 
\end{abstract}

\author{\sc Valentin Buciumas}
\address{Department of Mathematics\\ Pohang University of Science and Technology \\
Pohang 37673, Republic of Korea}
\thanks{}
\email{buciumas@postech.ac.kr}

\author{\sc Chun-Ju Lai}
\address{Institute of Mathematics \\ Academia Sinica \\ Taipei 106319, Taiwan} 
\email{cjlai@gate.sinica.edu.tw}


\maketitle

\setcounter{tocdepth}{1}
\tableofcontents

\section{Introduction}

\subsection{On $p$-adic groups, Hecke algebras and Whittaker modules}
\subsubsection{Hecke algebras}
Let $F$ be a local non-Archimedean field with ring of integers $\cO$, maximal ideal $\varpi \cO$ for a fixed uniformizer $\varpi$ and residue field $\kappa$ of cardinality $q$. Denote by $\redp:\cO\to \kappa$ the reduction mod $\varpi$ map. 
Let $G$ be the $p$-adic group $\GL_d(F)$.
The category $\Rep G$ of complex smooth admissible representations of $G$, of utmost importance in the local and global Langlands programme, can be studied in terms of certain Hecke algebras. 
Given a compact subgroup $X$ of $G$, the Hecke algebra $C_c^\infty(X \backslash {G} / X)$ captures information about all irreducible representations of $G$ with $X$-fixed vectors. 
More precisely, the functor $V \mapsto V^X$ gives a bijection between irreducibles of $G$ with $X$-fixed vector and irreducibles of $C_c^\infty(X \backslash {G} / X)$.
Important examples of Hecke algebras are the spherical, Iwahori and pro-$p$ Hecke algebras where $X$ is taken to be $K := \GL_d(\cO)$, $I := \redp^{-1} (\mathbf{B}(\kappa)) \subset K$ and $\Ip := \redp^{-1} (\mathbf{U}(\kappa))\subset K$, resp. 
Here, $\mathbf{B}$ and $\mathbf{U}$ denote the standard Borel and unipotent of $\GL_d$, resp. and $\redp:K \to \GL_d(\kappa)$ is the mod $\varpi$ reduction map.
 
The irreducible representations of $G$ with $K$-fixed vectors are called unramified; they appear at almost all places in an automorphic representation, and as such they are important in the theory of automorphic forms. 
The irreducible representations of $G$ with $I$-fixed vectors are called tamely ramified. 
They generate the principal block in the Bernstein decomposition of $\Rep G$~\cite{B76} (more intricately, by the work of Bushnell and Kutzko~\cite{BK93}, the affine Hecke algebra can actually be used to model the whole category $\Rep G$).   
The pro-$p$ subgroup $\Ip$ sees more than the Iwahori subgroup, including a certain class of depth $0$ representations.
Its Hecke algebra~\cite{V05,V16} is the natural object to consider when studying the representation theory of $G$ over a field of positive characteristic leading to basic results in the mod $p$ Langlands programme including classification of irreducibles in terms of irreducible supersingular representations of Levi subgroups~\cite{AHHV}. 

\subsubsection{Whittaker modules}
Let $\psi: U^-\to \CC^\times$ be an additive non-degenerate character of the negative unipotent $U^-$, with conductor $\varpi \cO$.
It is known that every irreducible $\pi \in \Rep G$ has an at most unique Whittaker model, in other words 
\begin{equation}\label{eq:multiplicityone}
	\dim \Hom_G(\pi, \Ind_{U^{-}}^G \psi ) \leq 1.
\end{equation}
Such a local result is fundamental in the theory of automorphic forms, where it can be used to show global multiplicity one results (for example that Fourier coefficients of automorphic forms on $\operatorname{GL}_d$ can be decomposed in terms of Whittaker functions~\cite{S74}), as well as in the study of L-functions, including in the Langlands--Shahidi method and the Rankin--Selberg formula. 

Alternatively, we can use the Gelfand--Graev module $\ind_{U^-}^{G} \psi$ of $G$ (the dual of $\Ind_{U^-}^{G} \psi^{-1}$), to derive information about Whittaker models for representations of $G$.
Restricting to a certain class of representations, we may consider the actions 
\begin{equation}\label{eq:HeckeGGaction}
C_c^\infty(K \backslash {G} / K) \curvearrowright (\ind_{U^-}^{G} \psi)^K, \quad 
\HGI := C_c^\infty(I \backslash {G} / I) \curvearrowright (\ind_{U^-}^{G} \psi)^I. 
\end{equation}
It is known that $(\ind_{U^-}^{G} \psi)^K$ is a rank one free module over $C_c^\infty(K \backslash {G} / K)$ (both of which are isomorphic to $\CC[\Rep G^\vee]$, first as an algebra, second as a module, via the Satake isomorphism and geometric Casselman--Shalika formula, resp.). 
Here, $G^\vee$ is the Langlands dual group, over $\CC$. 
It is also known that $(\ind_{U^-}^{G} \psi)^I$ is isomorphic to the antispherical module of the affine Hecke algebra $\HA\cong\HGI$~\cite{CS18}:
\begin{equation}
	(\ind_{U^-}^{G} \psi)^I \cong \operatorname{asph}:=\Ind_{\cH_\kappa}^{\HA} \operatorname{sgn},
\end{equation}
where $\cH_\kappa$ is the finite Hecke algebra.

\subsubsection{$p$-adic Schur algebras}\label{subsub:padicschur}
The space $(\ind_{U^-}^{G} \psi)^I$ can be seen as a module over the $p$-adic Schur algebra $\SGI:=\End_{C_c^\infty(I \backslash {G} / I)} (\ind_{U^{-}}^G \psi)^I$. 
Moreover, one can relate the space in~\eqref{eq:multiplicityone} to an irreducible representation of $\SGI$.  
The commutativity of $\SGI$ then implies the local multiplicity one property mentioned above.
It therefore makes sense to ask the following problems at the pro-$p$ level: 
\begin{enumerate}[label=($\textsf{P}$\arabic*):]
	\item understand the structure of $\WIp:= (\ind_{U^-}^{G} \psi)^{\Ip}$ as a module over $\HGIp:=C_c^\infty(\Ip \backslash {G} / \Ip)$.
	\item describe this structure in terms of an antispherical module of the $\HGIp$.
	\item understand the structure of the pro-$p$ Schur algebra  $\SGIp:=\End_{\HGIp} (\WIp)$. 
\end{enumerate} 

The action of $\HGIp$ on $\WIp$ is computed by Gao--Gurevich--Karasiewicz~\cite{GGK1}, therefore solving (\textsf{P}1).
Based on their result, we answer (\textsf{P}2)  and (\textsf{P}3) by using the theory of quantum wreath product developed in~\cite{LNX24,LM25,LNX25}. 

\subsubsection{Metaplectic groups}
Another natural appearance of the pro-$p$ Hecke algebra can  be seen when dealing with metaplectic covers of $p$-adic groups. 
Fix a positive integer $n$ and denote by $\mG$ a metaplectic $n$-cover of $G$ (see~\S\ref{sub:metaplecticgroups} for a precise definition; we note here that we are working with Savin's nice cover of $G=\mathbf{GL}_d(F)$). 
Our results are dependent on choosing this cover and do not immediately generalize to other covers (like the Kazhdan--Patterson covers) or other Cartan types cf. results in the same style proved in~\cite{BBB19,GGK2}.

In this setting, the multiplicity one property for Whittaker models fails, i.e.~\eqref{eq:multiplicityone} fails when we replace $G$ by $\mG$. 
As such, understanding the structure of the modules in~\eqref{eq:HeckeGGaction} is a significantly harder problem. 
Hecke algebras and their modules will consist of genuine functions, which coupled with the lack of commutativity of the metaplectic torus reduce the support of the Hecke algebra, while leaving intact the support of the corresponding module.  
For example, in the metaplectic setting the spherical Gelfand--Graev module is not free of rank one over the spherical Hecke algebra and its structure is quite hard to understand $p$-adically (cf. \cite{BP25}).

Gao--Gurevich--Karasiewicz~\cite{GGK1} have realized that the reduction in support of the Hecke algebra characteristic to the spherical and Iwahori levels does not persist at the pro-$p$ level, which allowed them to algebraically compute the action of $\HGbIp:= C_{\iota, c}^\infty(\Ip \backslash {\mG} / \Ip)$ on the metaplectic module $\WbIp:=(\ind_{\mu_n U^-}^{\mG} \iota \otimes \psi)^{\Ip}$ (thereby solving (\textsf{P}1) at the metaplectic level). 
They use an averaging procedure to compute the Iwahori level action.  
Ultimately, the structure at the spherical level is understood in~\cite{BP25} in terms of dual data depending on Lusztig's quantum group at a root of unity.

\subsection{On quantum wreath products and wreath modules}

\subsubsection{Quantum wreath Products}
A \emph{quantum wreath product} (introduced in \cite{LNX24}) is an algebra  $B \wr \cH(d)$
defined by a generalized Bernstein--Lusztig type presentation.
The notion of quantum wreath product allows one to uniformly study representation theory of many important algebras appearing in quantum representation theory such as the affine Hecke algebra, (affine) Yokonuma algebra, affine zigzag algebras, affine Frobenius algebras, etc.  

In group theory, a wreath product combines the ``base group'' and the ``acting group'' into a larger group. 
However, in the context of quantum wreath products, 
one cannot encompass many important examples by taking the wreath product of two ``quantum'' algebras.
What we do is, given a $\CC$-algebra $B$ that is potentially non-commutative, we define an algebra $B \wr \cH(d)$  generated by $B^{\otimes d}$ and formal generators $H_1, \dots, H_{d-1}$, 
subject to the type A braid relations among $H_i$'s, 
quadratic relations for $H_i$'s with coefficients in $B^{\otimes d}$, 
and wreath relations between $H_i$'s and $B^{\otimes d}$,
which is meant to mimic the wreath product between $B^{\otimes d}$ and the finite Hecke algebra. 
But we stress that $\cH(d)$ is not defined as the finite Hecke algebra (in fact it is not defined independently of $B \wr \cH(d)$).
The coefficients appearing in the definition of $B \wr \cH(d)$ are controlled by a choice of parameters $(S,R,\sigma,\rho)$, where $S,R \in B\otimes B$ and $\sigma, \rho \in \End(B\otimes B)$.
The theory becomes substantially more interesting when $S,R \not\in\CC(1\otimes 1)$.

For example, setting $B= \CC[x^{\pm1}]$ to be a polynomial ring produces an algebra $B \wr \cH(d)$ isomorphic to the affine Hecke algebra $\HA$ (or equiv., the Iwahori Hecke algebra $\HGI$), while taking $B= \Fr[x^{\pm1}]$ where $\Fr:=\CC[C_{q-1}]$ is the group algebra over the cyclic group of $q-1$ elements produces the affine Yokonuma algebra $\HAY$ (equiv., the pro-$p$ Hecke algebra $\HGIp$).  
Moreover, one can set $B = \cH_q(\Sigma_m)$ to be the Hecke algebra for the symmetric group $\Sigma_m$ to realize the so-called Hu algebra, a non-trivial quantization of wreath products between symmetric groups. 
Furthermore, the base algebra $B$ does not have to be deformed from a group algebra. The affine zigzag algebra is realized by setting $B$ to be the zigzag algebra over a given quiver.

\subsubsection{Skew polynomial quantum wreath products}
Generalizing results in \cite{LM25}, we develop the theory of skew polynomial quantum wreath products where the base algebra is a ring of skew polynomials. 
In particular, we work with the ring $\bbA := \Fr \rtimes_n \CC[x^{\pm1}]$ with skew commutativity 
\[
x^{\pm1} f = \xi^{\pm2} f x^{\pm 1},
\quad
\textup{ for } f \in \Fr,
\]
for some $n$-th root of unity $\xi$.
Fix $t \in C_{q-1}$ so $\Fr \cong \CC[t]/(t^{q-1}-1)$.
We can then define the quantum wreath product $\bbA\wr\cH(d)$ using the same choice of parameters $S, R$, and $\sigma$ as in the affine Yokonuma algebra $\Fr[x^{\pm1}]\wr\cH(d)$,
while the map $\rho$ is determined by the twisted Leibniz rule $\rho(ab) = \sigma(a)\rho(b) + \rho(a)b$ for all $a,b\in\bbA$ as well as the following ``minuscule'' case:
\[
\rho(x\otimes 1) := \sum_{j=1}^{q-1}(t^j \otimes t^{-j})(x\otimes 1), 
\quad
\rho(\Fr\otimes \Fr) = 0.
\]
We show that the intricate terms in the Bernstein--Lusztig relations \eqref{eq:MpWR} for $\HGbIp$ are captured by the values $\rho(P)$, where $P \in \bbA \otimes \bbA$ is an arbitrary skew polynomial. This formulation provides a more transparent realization of the relations originally established in~\cite{GGK1}.

Our first main result identifies the metaplectic pro-$p$ Hecke algebra in terms of the skew quantum wreath product:
\begin{mainthm}[Theorem \ref{thm:PQWP2}]\label{thmPQWP}
	We have an isomorphism of algebras $\HGbIp \cong \bbA \wr \cH(d).$
\end{mainthm}

Our result allows us to apply the structure and representation theory developed in \cite{LNX24, LM25} to $\HGbIp$. 
In particular, we obtain a PBW basis theorem for $\HGbIp$
(see Theorem \ref{thm:SPQWPPBW}).

When $n=1$, $\bbA \wr \cH(d)$ is known to be isomorphic to the affine Yokonuma algebra $\HAY$ (cf.~\cite{LM25}) and we therefore recover the main result of~\cite{CS16} identifying the latter algebra with the pro-$p$ Hecke algebra.
A PBW basis result for $\HAY$ is obtained in \cite{RS20}.
Representation theoretic results such as a Schur duality for $\HAY$ were developed in~\cite{LM25}.

\subsubsection{Descent to the Iwahori level}
We can write
\begin{equation}\label{eq:corner}    
\HGI \cong e_I \HGIp e_I, \quad \HGbI \cong e_I \HGbIp e_I.
\end{equation}
for a certain idempotent $e_I$ in $\HGIp$ (or $\HGbIp$) that is supported on $I$. 
As mentioned before, $\HGI \cong \CC[x^{\pm1}]\wr\cH(d).$
The metaplectic Hecke algebra $\HGbI := C^\infty_{\iota, c}(I \backslash {\mG} / I)$ is known to be isomorphic to the non-metaplectic Hecke algebra $\HGI$ (for the cover we are working with) cf. Savin's local Shimura correspondence~\cite{S88,McN12}. 

We show in Proposition \ref{thm:eHe} that $\CC[x^{\pm n}]\wr\cH(d)$, 
can be realized as a corner algebra $\epsilon_I (\bbA\wr\cH(d)) \epsilon_I$ using the same idempotent $\epsilon_I \in \Fr^{\otimes d}$ as in the linear case.
We obtain a quantum wreath product interpretation of the local Shimura correspondence, i.e. we show the existence of an isomorphism:
\[
\epsilon_I (\Fr[x^{\pm1}] \wr \cH(d)) \epsilon_I
\to
\epsilon_I (\bbA \wr \cH(d)) \epsilon_I
\quad \textup{which maps} \quad
\epsilon_I x^{\lambda}  \epsilon_I
\mapsto
\epsilon_I x^{\bar{n}\lambda}  \epsilon_I,
\]
where $\bar{n}:= n/\textup{gcd}(n,2)$.
In fact, our result clarifies that the analogous map $\Fr[x^{\pm1}]\wr \cH(d) \to \bbA \wr\cH(d)$ at the pro-$p$ level will no longer be an isomorphism 
since the centers of the two algebras (depending on  
$\Fr \rtimes_n \CC[x^{\pm1}]$ for $n\geq2$ and $\Fr[x]$) will be different. This gives us:
\begin{maincor}
	For $n \geq 2$, the pro-$p$ Hecke algebras $\HGbIp \not\cong \HGIp$ are non-isomorphic.
\end{maincor}
This result highlights the \emph{failure} of the local Shimura correspondence at the pro-$p$ level.

\subsubsection{Wreath modules}

Since there is no ``acting algebra'' $\cH(d)$ in a quantum wreath product, one cannot combine a $B$-module and an $\cH(d)$-module (which is not defined) into a $B\wr \cH(d)$-module.
A crucial result in \cite{LNX25} is the introduction of the notion of the wreath modules.
It turns out that one has to start with a certain $B$-module on which both $S$ and $R$ act on $M \otimes M$ by scalars (see \S\ref{sub:wreathmod} for details) and derives an auxiliary {\em twisted Hecke algebra} $\tcH(M)$ using these scalars. 
Subsequently, for any module $N$ over $\tcH(M)$, a well-defined $B\wr\cH(d)$-action can be realized on the wreath module $M \wr N := M^{\otimes d} \otimes N$.

This theory provides a transparent realization of a diverse family of modules in existing literature.
For example, the spherical, antispherical and Kashiwara--Miwa--Stern~\cite[(32)]{KMS95} tensor modules for the affine Hecke algebra can all be realized as wreath modules. 
Moreover, simple and Specht modules for Ariki--Koike and Hu algebras are constructed via parabolic induction of wreath modules.

\subsubsection{Wreath modules for skew polynomial quantum wreath products}
One significance of employing the theory of quantum wreath products is that we can now define a novel ``antispherical module'' for the quantum wreath product $\bbA\wr \cH(d)$. 

There is a notable technicality which arises from the fact that the quadratic relation in $\bbA\wr \cH(d)$ has a complicated degree one coefficient. 
This is a difficulty specific to the pro-$p$ level (it does not appear at the Iwahori level) and appears even in the non-metaplectic setting.
Let $t$ be a generator of $C_{q-1}$. 
The quadratic relations in $\bbA\wr \cH(2)$ are of the form
\begin{equation}\label{eq:quadratic:intro}
H^2 = (q-1)e H + q(t^k\otimes t^k),
\quad
\textup{where}
\quad
e := \frac{1}{m}\sum_{j=1}^m t^j \otimes t^{-j},
\quad 
k \in \{0, m/2\}.
\end{equation}
It is a pleasant surprise (see Lemma \ref{Qsplits2}) that such a quadratic relation splits into linear factors, i.e., $(H+\al)(H - \alb) = (H - \alb)(H+\al)=0$ for some $\al, \alb \in \Fr^{\otimes d}$.
To be precise, we obtain that 
\begin{equation}\label{eq:eigenvalues:intro}
\al = (\sqrt{q} (t^{k}\otimes 1) + 1)e - \sqrt{q}(t^{k}\otimes 1),
\quad
\alb = (\sqrt{q}(1\otimes t^{k}) + q)e - \sqrt{q}(1\otimes t^{k}).
\end{equation}
The splitting of~\eqref{eq:quadratic:intro} enables us to define the ``trivial'' and ``sign'' modules on which $H$ acts by $\alb$ and $-\al$, respectively.
We are then able to construct a generalization of the KMS space $\bbV(N)^{\otimes d}$ (see \cref{KMSaction}) for any $N>1$, and then show: 
\begin{mainthm}[Theorem \ref{thm:VGG2}]\label{thmKMS}
	The pro-$p$ component $\VGGb^{\Ip}$ of the metaplectic Gelfand--Graev module is isomorphic to the space $\bbV(1) \wr \sgn \cong  \bbV(1)^{\otimes d}$ 
	as right modules over $\HGbIp \cong \bbA\wr \cH(d)$.
\end{mainthm}
Here $\bbV(1) \wr \sgn$ is the new antispherical module of $ \bbA\wr \cH(d)$, generalizing the antispherical module of $\HA$ and giving us a solution to (\textsf{P}2) at the metaplectic level.
The $ \bbA\wr \cH(d)$ module $\bbV(N)\wr\sgn$ can then be interpreted as a generalization of the KMS tensor module~\cite[(32)]{KMS95}.
We stress that this result is new and non-trivial even in the linear case (i.e. when we work with $\operatorname{GL}_d$ instead of its cover).

Combining the corner algebra realization \eqref{eq:corner},
we also show that, as an $\HGbI$-module, $\WbI \cong \bbV(1)^{\otimes d}\epsilon_I$ is isomorphic to the original KMS tensor space $V(\bar{n})^{\otimes d}$ (see Proposition \ref{prop:KMSIwa}). 

\subsubsection{Pro-$p$ metaplectic Schur algebras}
We are now in a position to study the wreath Schur algebras
$\bbS(N,d):=\End_{\bbA\wr\cH(d)}(\bbV(N)^{\otimes d})$
at the pro-$p$ metaplectic level.
\begin{mainthm}
    \begin{enua}
\item (\cref{prop:mult2}) For $N \geq 1$, we obtain explicit multiplication formulas of $\bbS(N,d)$ with respect to the Dipper-James-type basis introduced in \cite{LM25}.
\item (Corollary \ref{cor:Schuriden}) For $N=1$, we establish an algebra isomorphism 
\begin{equation}\label{eq:isoSalgebra}
\bbS(1,d) \cong \SGbIp:=\End_{\HGbIp}(\WbIp),
\end{equation}
which solves (\textsf{P}3) at the metaplectic level.
    \end{enua}
\end{mainthm}
Descending to the Iwahori level, we can use the fact that 
$\WbI \cong V(\bar{n})^{\otimes d}$ to write the Iwahori--Schur algebra $\SGbI:= \End_{\HGbI}(\WbI)$ as the quantum affine Schur algebra $\widehat{S}_{\sqrt{q}}(\bar{n},d) \cong \End_{\cH_{\operatorname{aff}}} (V(\bar{n})^{\otimes d})$ \cite{G99}.
This recovers one of the main results of~\cite[\S4]{GGK2} and was an important motivation for our work. 

The Schur algebra $\SGbIp$, even in the non-metaplectic setting, is more complicated; it is not commutative, and contains a lot more information than before.
For example, it is slightly more complicated than the Iwahori--Schur algebra $\SGbI \cong \widehat{S}_{\sqrt{q}}(\bar{n},d)$ for a \emph{metaplectic group} cf. our main results.

The structure and representation theory of the quantum affine Schur algebra developed into an elegant mathematical framework with many applications to the study of $p$-adic groups~\cite{V03, MS19, G22, GGK2}, including to the study of Whittaker models as mentioned in \S\ref{subsub:padicschur}. 
We hope this paper is the first step in understanding similar results for the newly introduced pro-$p$ Schur algebras. 
For example, we expect that one can use the theory of quantum wreath products to show that the multiplicative coefficients of these pro-$p$ Schur algebras $\bbS(N,d)$ stabilize as $d \to \infty$, just as for the family of quantum affine Schur algebras.
Such a property is the key to the close connection to the theory of quantum affine algebras.

\subsection*{Acknowledgements} 
We would like to thank Vincent S\'echerre for explaining the work~\cite{CS16}, Edmund Karasiewicz for explaining the work~\cite{GGK1}, Nadya Gurevich, Edmund Karasiewicz, Alexandre Minets and Travis Scrimshaw for general discussions. 

Research of the first author was supported by the Basic Science Research Institute Fund, National Research Foundation of Korea (NRF) grant 2021R1A6A1A10042944 and a HKUST-POSTECH Joint Research Seed Grant.
Research of the second author was supported in part by
NSTC grants 114-2628-M-001-001 and the National Center of Theoretical Sciences.

\section{Polynomial Quantum Wreath Products}\label{sec:PQWP}
In this section, we recall the theory of polynomial quantum wreath products developed in \cite{LM25}.
We explicitly identify the affine Yokonuma algebra and its generalizations with polynomial quantum wreath products.
We introduce generalizations of the trivial and sign representations for the finite Yokonuma algebra which allows us to define for the affine Yokonuma algebra the (anti)spherical module and the KMS tensor module which generalizes~\cite[(32)]{KMS95}.
Ultimately, these will be identified with the Gelfand--Graev module of a $p$-adic group. 
We define and study the corresponding Schur algebra.

\subsection{Polynomial quantum wreath products}
Let $m = q-1 \in \ZZ_{\geq 1}$,
and let $C_m$ be the cyclic group of order $m$ generated by $t$.
Define $[m]:=\{0,1,\cdots, m-1\}$.
Let
\begin{equation}\label{eq:def:FB}
\Fr := \bbC[C_m] = \tfrac{\bbC[t]}{(t^m-1)} = \textup{Span}_\CC\{t^j ~|~ j \in [m]\},
\qquad
\bbB :=  \Fr[x^{\pm1}].
\end{equation}
For any $b\in \bbB^{\otimes k}$ for $k\in\{1,2\}$ and $\varphi \in \End_{\CC}(\bbB^{\otimes 2})$, we define $b_i\in \bbB^{\otimes d}$ and $\varphi_i \in \End_{\CC}(\bbB^{\otimes d})$ by
\begin{equation} \label{def:subscriptshorthand}
\begin{split}
&b_i := 1^{\otimes i-1} \otimes b \otimes 1^{\otimes d-i-k+1} ,
\\
&\varphi_i(b_1 \otimes \dots \otimes b_d) := b_1 \otimes \dots \otimes b_{i-1}\otimes \varphi(b_i\otimes b_{i+1}) \otimes \dots \otimes b_d.
\end{split}
\end{equation}
In particular, $\bbB^{\otimes d} \cong \Fr^{\otimes d}[x_1, \dots, x_d]$.
Let $e := \frac{1}{m} \sum_{j=1}^m t^j \otimes t^{-j}  \in \Fr\otimes \Fr$. Thus, for $1 \leq i \leq d-1$: 
\begin{equation}\label{def:e}
e_i: = \frac{1}{m} \sum\nolimits_{j=1}^m (1^{\otimes i-1}\otimes t^j \otimes t^{-j} \otimes 1^{\otimes d-i-1}) 
= \frac{1}{m} \sum\nolimits_{j=1}^m t_i^j t_{i+1}^{-j} 
\in \Fr^{\otimes d}.
\end{equation}
Let $\sigma = \textup{flip} \in \End_\bbC(\bbB^{\otimes 2})$ be the flip map $b_1\otimes b_2 \mapsto b_2 \otimes b_1$,
and let $\partial\in \End_\bbC(\bbB^{\otimes 2})$ be the Demazure operator given by 
$\partial( b ) := \tfrac{b - \sigma(b)}{x_1 - x_2}$.
In particular, $\partial_i( b ) := \tfrac{b - \sigma_i(b)}{x_i - x_{i+1}}$.
\begin{prop}\label{DemazureFacts}
Let $S_i := (q-1) e_i = \sum_{j=1}^m t_i^j t_{i+1}^{-j}\in \Fr^{\otimes d} \subseteq \bbB^{\otimes d}$.
\begin{enua}
\item Each element $e_i \in \Fr^{\otimes d}$ is an idempotent, i.e., $e_i^2 = e_i$.
\item $S_i f = \sigma_i(f) S_i$ for all $f \in \Fr^{\otimes d}$.
\item The following twisted Demazure operator $\rho_i \in \End_\bbC(\bbB^{\otimes d})$ is well-defined for each $i$:
\begin{equation}\label{def:tDemazure}
\rho_i( b ) := \frac{S_i x_i b - \sigma_i(b) S_i x_i}{x_i - x_{i+1}}.
\end{equation}
\item Let $f \in \Fr^{\otimes d}$ and $P \in \bbC[x_1^{\pm1}, \dots, x_d^{\pm1}]$. Then, 
$\rho_i( fP ) = \sigma_i(f) \partial_i(P) S_i x_i$. 
In particular, $\rho_i$ is uniquely determined by the value $\rho_i(x_i) = S_ix_i$.
\end{enua}
\end{prop}
\begin{proof}
Parts (a) and (b) follow from a direct verification.
Part (c) follows from (b), and part (d) follows from another direct calculation, see \cite[(3.1)]{LM25}.    
\end{proof}
\begin{defn}\label{def:PQWP}
	Denote by $\Fr[x^{\pm1}] \wr \cH(d)$ the 
	{\em polynomial quantum wreath product} defined in \cite[Definition 3.7]{LM25} for the choices $\Fr := \bbC[C_m]$, $\bbB = \Fr[x^{\pm1}]$, 
	$R = q(t^k \otimes t^k)$ for some $0\leq k < m$
	and $\beta = (q-1) e x_1$.
	In other words, $\Fr[x^{\pm1}] \wr \cH(d)$ denotes a $\CC$-algebra generated by $H_1, \cdots, H_{d-1}$ and elements in $\bbB^{\otimes d}$, subject to the following relations:
	\begin{align}
		&\textup{(braid relations) } \label{def:BR}
		& H_l H_{l+ 1} H_l = H_{l+ 1} H_l H_{l + 1},  \quad H_i H_j = H_j H_i,
		\\
		&\textup{(quadratic relations) }\label{def:QR}
		&H_i^2 = {S_i}H_i + R_i ,
		\\
		&\textup{(wreath relations) }\label{def:WR}
		& H_ib = \sigma_i(b)H_i + \rho_i(b),
	\end{align}
	for $1\leq l \leq d-2$, $1\leq i \leq d-1$, $|j-i|\geq 2$, and $b\in \bbB^{\otimes d}$.
	Here, the choice of parameters are given by
	\begin{equation}\label{eq:PQWPpar}
		\sigma = \textup{flip},
		\quad
		S = (q-1)e,
		\quad
		\rho(x\otimes 1) = S (x \otimes 1),
		\quad
		R  = q(t^k \otimes t^k).
	\end{equation}
\end{defn}

In particular, \eqref{def:WR} implies the following relations, thanks to \cref{DemazureFacts}(d):
\begin{align}
    \label{eq:commutationHf} H_if = \sigma_i(f)H_i,  &\qquad ( f\in \Fr^{\otimes d}),
    \\
    \label{eq:commutationHP} H_iP = \sigma_i(P)H_i + \partial_i(P) S_i x_i, &\qquad ( P \in \bbC[x_1^{\pm1}, \dots, x_d^{\pm1}]).
\end{align}

\begin{expl}\label{ex:aha}
	Consider the special case when $m=1$, i.e. $\Fr=\CC$. The quantum wreath product $\CC[x^{\pm1}] \wr \cH(d)$ has quadratic and wreath relations:
		\[
		H_i^2 = (q-1) H_i + q,
		\quad
		H_i b = \sigma_i(b)H_i + (q-1) \frac{b - \sigma_i(b)}{x_i - x_{i+1}} x_i,
		\quad
		(b \in \CC[x^{\pm1}]^{\otimes d}).
		\]
		It then follows that
		$\CC[x^{\pm1}] \wr \cH(d):= \HA$ is  the affine Hecke algebra.
		See \cite[\S4.2]{LNX24} for details.
\end{expl}
\begin{expl}\label{ex:maha}
	Fix $n\geq 1$. 
	One can consider another quantum wreath product $\CC[x^{\pm n}] \wr \cH(d)$ that is technically not an instance of the polynomial one defined in \cref{def:PQWP}.
	Here, the relations are: 
	\[
	H_i^2 = (q-1) H_i + q,
	\quad
	H_i b = \sigma_i(b)H_i + (q-1) \frac{b - \sigma_i(b)}{x_i^n - x^n_{i+1}} x^n_i,
	\quad
	(b \in \CC[x^{\pm n}]^{\otimes d}).
	\]
    Note the denominator in the above wreath relation is $x_i^n-x_{i+1}^n$ and not $x_i - x_{i+1}$. 
	We will see in Proposition~\ref{prop:IHtoPQWP} that $\CC[x^{\pm n}] \wr \cH(d) \cong \HGbI$, the metaplectic Iwahori-Hecke algebra. 
	The assignment $x_i^{\pm1} \mapsto x_i^{\pm n}$, $H_i \mapsto H_i$ gives an algebra isomorphism between $\CC[x^{\pm n}] \wr \cH(d)$ and the algebra in Example~\ref{ex:aha}.
	$p$-adically, 
	the isomorphism $\HGI \cong \HGbI$ is known as the local Shimura correspondence~\cite{S88}.
\end{expl}

In \S\ref{subsec:HAY} we will consider the affine Yokonuma Hecke algebra $\HAY$ defined in \cite[Section 2]{CP16} as an example of polynomial quantum wreath product. 
$\HAY$ is a deformation of the group algebra of the wreath product $(C_m \times \bbZ) \wr \Sigma_d$.
It was first observed in \cite{RS20} as an instance of a Rosso--Savage algebra, and hence $\HAY$ can be realized as a polynomial quantum wreath product with $R = q(1\otimes 1)$.

\begin{rmk}\label{rmk:gldsld}
For our purpose, we prefer to work with a more general setup where $R = q(t^k\otimes t^{k})$ for some $0\leq k < m$. 
This choice gives more flexibility and is useful in practice. 
While one of the main objectives of this work is to relate this quantum wreath product to the pro-$p$ Iwahori Hecke algebra of $\GL_d$ (and its metaplectic cover) where we will specialize $R$ to $q (1\otimes 1)$, our result can be generalized to $\mathbf{SL}_d$, where one should work with $R=q(t^{m/2} \otimes t^{m/2})$ instead (cf.~\eqref{eq:MpQR} and the matching in Theorem~\ref{thm:PQWP2}, see also~\cite[p. 715]{CS16}).
\end{rmk}

In the following, we will briefly discuss the special case $R = q(1\otimes 1)$ in Section \ref{subsec:HAY}, and then deal with the general case in the rest of Section \ref{sec:PQWP}.

\subsubsection{Affine Yokonuma Hecke Algebras} \label{subsec:HAY}

While $q$ can be generic in \cite{CS16}, here we focus on the case when $q$ is the order of a finite field, i.e., $q \in \ZZ_{\geq 2}$.  
Consider the polynomial quantum wreath product $\HAY \cong \Fr[x^{\pm1}]\wr\cH(d)$ for the choices in \eqref{eq:PQWPpar}, with 
\begin{equation}\label{def:R-Yokonuma}
	R = q(1\otimes 1).
\end{equation}
It is surprising that \eqref{def:QR}, although involving idempotents in $\Fr^{\otimes d}$, does split as follows:
\begin{lem}\label{Qsplits}
For each $i$, the equation 
$(H_i + \al_i) (H_i - \alb_i) = 0 = (H_i - \alb_i)(H_i + \al_i) $
holds in $\HAY$, 
where $\al_i, \alb_i \in \Fr^{\otimes d} \subseteq \bbB^{\otimes d}$ are given by:
\begin{equation}\label{def:alphab}
\al_i := (\sqrt{q}+1) e_i -\sqrt{q}(1^{\otimes d}),
\quad
\alb_i := (\sqrt{q}+q)e_i -\sqrt{q}(1^{\otimes d}).
\end{equation}
\end{lem}
The proof of Lemma \ref{Qsplits} follows from a direct computation.
It can also be proved under a more general setup (see Lemma \ref{Qsplits2}), so we omit the calculation here.

\begin{rmk}
    Note that the sign convention is chosen so that 
\[
(H_i+\al_i) H_i = (H_i + \al_i) \alb_i
\quad
\textup{and}
\quad
(H_i -\alb_i)H_i = (H_i - \alb_i)\al_i,
\]
i.e., 
 $\alb_i$ is an $H_i$-eigenvalue with respect to the eigenvector $H_i + \al_i$, 
 and  $-\al_i$ is an $H_i$-eigenvalue with respect to the eigenvector $H_i - \alb_i$.
\end{rmk}
\begin{rmk}
Since the braid relations hold, elements of the form $H_w$ for $w\in \Sigma_d$ are well-defined.
Recall from \cref{ex:aha} that $\CC[x^{\pm1}]\wr \cH(d)$ is isomorphic to the affine Hecke algebra $\HA$.
In this case, the subspace $\Span_\bbC\{ H_w ~|~w\in \Sigma_d\} \subseteq \HA$ forms a $\bbC$-subalgebra that is isomorphic to the finite Hecke algebra.

In the Yokonuma case, the subspace $\Span_\bbC\{ H_w ~|~ w\in \Sigma_d\} \subseteq \HAY$ is not closed under multiplication; as such the finite (or affine) Hecke algebra is not a subalgebra of the affine Yokonuma algebra. 

\end{rmk}
\begin{rmk}\label{rem:finite-affine-yokonuma}
We note that the finite Yokonuma Hecke algebra $\HY$ can also be realized as a quantum wreath product $\HY \cong \Fr \wr  \cH(d)$.
The wreath relation \eqref{def:WR} becomes $H_i f = \sigma_i(f) H_i$ for $f \in \Fr^{\otimes d}$, 
since it follows from \eqref{def:tDemazure} and $\partial_i(P) = 0$ that $\rho_i(f) = 0$. 
$\HY$ is then generated by elements $H_1, \cdots H_{d-1}$ and elements in $\Fr^{\otimes d}$ and can easily be seen to be a subalgebra of $\HAY$. 
\end{rmk}
\subsubsection{The Splitting Lemma}\label{subsec:HGIp}
Within this section, consider the quantum wreath product $\Fr\wr\cH(d)$ for the choices in \eqref{eq:PQWPpar}, with  $R = q(t^{k}\otimes t^{k})$ for some $0 \leq k \leq m-1$.
When $R= q(t^{k}\otimes t^{k})$ is not a scalar (i.e. $k \neq 0$), the existence of an element $\sqrt{R} \in \Fr \otimes \Fr$ that squares into $R$ is not guaranteed, 
and hence one cannot always renormalize the quadratic relations into ones taking the form $T^2 \in 1 + (\Fr \otimes \Fr) T$.

We begin with a lemma regarding interaction with the idempotent $e$.
\begin{lem}\label{lem:afc}
\begin{enua}
\item The element $e$ is an intertwiner. That is, $e(f\otimes f') = (f'\otimes f)e$ for all $f,f' \in \Fr$.
\item The following hold: $(t\otimes 1) e = e (t\otimes 1) = (1\otimes t) e = e (1\otimes t)$.
\item The set $\{ e, (t\otimes 1) e, \dots, (t^{m-1}\otimes 1)e\}$ forms a $\CC$-basis of $(\Fr\otimes\Fr) e$.
In particular, any element in $(\Fr\otimes\Fr)e$ must take the form $ae$ for some $a \in \Fr\otimes \CC$.
\end{enua}
\end{lem}
\begin{proof}
Part (a) is a special case of \cite[Lemma 2.6]{WW08}.
Part (b) follows from the commutativity of $\Fr = \CC[t]/(t^m-1)$ and part (a).
Part (c) follows from part (b).
\end{proof}
Next, we prove the splitting lemma.
\begin{lem}\label{Qsplits2}
Let $\al = ae+b$ for some $a \in \Fr\otimes \CC$, $b \in \Fr\otimes \Fr$,
and let $\alb := \sigma(\al) + (q-1) e$.
Then, the equation
$(H_i+\al_i)(H_i - \alb_i) = 0 = (H_i - \alb_i)(H_i+\al_i)$ 
holds in $\Fr\wr\cH(d)$,
for any of the following choices:
\[
b\in \{\pm \sqrt{q}(t^i\otimes t^{k-i})~|~0\leq i \leq k\},
\quad
a\in\{-b-q, -b+1\}.
\]
In particular, one has $(H_i+\al_i)(H_i - \alb_i) = 0 = (H_i - \alb_i)(H_i+\al_i)$  where:
\begin{equation}\label{def:alphab2}
    \al_i = (\sqrt{q} t^{k}_i + 1)e_i - \sqrt{q}t^{k}_i,
\quad
\alb_i = (\sqrt{q}t^{k}_{i+1} + q)e_i - \sqrt{q}t^{k}_{i+1}.
\end{equation}
\end{lem}
\begin{proof}
First, we show that $(H_i+\al_i)(H_i - \alb_i) = 0$ is equivalent to $(H_i - \alb_i)(H_i+\al_i) = 0$ for any $\al \in \Fr \otimes \Fr$, equivalently,
\[
\al H - H \alb  = H \al  -  \alb H,
\quad 
\al \alb = \alb \al.
\]
The first equality holds since $\al - \sigma(\alb) = \al - \sigma^2(\al) + \sigma(S) = S = \sigma(\al) - \alb$.
The second equality follows from the fact that $\Fr$ is commutative.

Next, it follows from Lemma \ref{lem:afc}
\begin{equation}\label{eq:calcidem}
ae = \sigma(a)e,
\quad
b e = e \sigma(b),
\quad
b^2 e = be\sigma(b)  = b\sigma(b) e.
\end{equation}
For simplicity, we may assume that $d=2$ and let $H:= H_1$. Now,
\[
(H+\al)(H - \alb) = H^2 - (\sigma(\alb) - \al)H - \al \alb.
\]
Since $\sigma(\alb) = \al + (q-1)e$, the degree one coefficient is precisely $S$. 
It remains to determine the values of $a$ and $b$ such that $R = \al \alb$. 
Expanding $\al \alb$ and applying \eqref{eq:calcidem}, we get
\[
\begin{split}
q(t^k\otimes t^k) &= (ae+b)(\sigma(a)e+\sigma(b) + (q-1)e)
\\
&= ae\sigma(a)e+ (q-1)ae^2 +ae\sigma(b)
+b\sigma(a)e+ (q-1)be +b\sigma(b)
\\
&= (a^2+ (2b+(q-1)) a + (q-1)b)e + b\sigma(b).
\end{split}
\]
That is, $b\sigma(b) = q(t^k\otimes t^k)$ and hence $b\in \{\pm \sqrt{q}(t^i\otimes t^{k-i})~|~0\leq i \leq k\}$.
Moreover,
\[
    0 = (a^2+ (2b+(q-1)) a + (q-1)b)e
= ((a+b+\tfrac{q-1}{2})^2  -(b+\tfrac{q-1}{2})^2 +(q-1)b)e.
\]
Note that
$(-(b+\tfrac{q-1}{2})^2 +(q-1)b)e
=
(-b^2 - (\tfrac{q-1}{2})^2 )e
$. It follows from \eqref{eq:calcidem} that
\[
-b^2e - (\tfrac{q-1}{2})^2e
=
- b\sigma(b)e - \tfrac{q^2 - 2q + 1}{4}e
= 
-(\tfrac{q + 1}{2})^2e,
\]
which leads to the fact that
\[
0 =
((a+b+\tfrac{q-1}{2})^2 -(\tfrac{q + 1}{2})^2)e
=
(a+b+q)(a+b-1)e.
\]
Therefore, $a \in \{-b-q,-b+1\}$. 
\end{proof}
\subsection{The Kashiwara--Miwa--Stern tensor space} \label{sec:KMSforPQWP}
For the rest of this section, consider the polynomial quantum wreath product $\bH:=\Fr[x^{\pm1}]\wr\cH(d)$ for the choices in \eqref{eq:PQWPpar}, with  $R = q(t^{k}\otimes t^{k})$ for some $0 \leq k \leq m-1$.
In this section, we describe the analogs of the KMS tensor space~\cite[(32)]{KMS95} for any such $\bH$ 
, which includes the affine Yokonuma Hecke algebra $\HAY$.

Fix a positive integer $N$.
Let $\bbV$ be the free right $\Fr$-module with basis $\{v_i\}_{i \in \bbZ}$.
Let $x^{\pm1}$ act on $\bbV$ by $v_i \cdot x^{\pm 1} = v_{i\pm N}$ for all $i$, and hence $\bbV = \bbV(N)$ carries a structure of $B$-module for any $N$.

Next, the $\Fr^{\otimes d}$-module $\bbV^{\otimes d}$ has a basis 
$\{v_f:= v_{f(1)} \otimes \cdots \otimes v_{f(d)}| f: [d]\to \bbZ\}$, on which $\Sigma_d$ acts (from the right) by place permutation.
Denote by $s_i := (i~i+1) \in \Sigma_d$ the simple transposition. 
\begin{prop}\label{KMSaction}
Recall elements $e_i$ and $\al_i$ in $\Fr^{\otimes d}$ from  \eqref{def:e} and \eqref{def:alphab2}. 
Then:
\begin{enua}
\item
There is a unique right $\bH$-module structure on $\bbV^{\otimes d}$ determined by the following:
\begin{equation}\label{eq:fundaction}
  v_f \cdot H_i = \begin{cases}
    v_{f\cdot s_i}
    &\textup{if } f(i) < f(i+1),
    \\
     v_{f}(-\al_i) 
     &\textup{if }  f(i) = f(i+1),
    \\
    v_{f\cdot s_i} R_i
    + v_{f} (q-1) e_i 
    &\textup{if }  f(i) > f(i+1),
    \end{cases}
\end{equation}
where $f$ is in the fundamental region, i.e. $f(i) \in \{1, \dots, N\}$ for all $i$.
\item  There is a decomposition into $\bH$-submodules as follows:
\[
\bbV^{\otimes d} \cong \bigoplus\nolimits_{\lambda \in \Lambda_{N,d}} M^\lambda,
\]
where $\Lambda_{N,d}$ consists of all weak compositions of $d$ of $N$-steps,
and each submodule $M^\lambda$ is generated from the basis element 
$v_\lambda 
:= v_{1}^{\otimes \lambda_1} \otimes 
v_{2}^{\otimes \lambda_2} \otimes 
\dots
v_{N}^{\otimes \lambda_N} \in \bbV^{\otimes d}$.
\end{enua}
\end{prop}
\begin{proof}
Part (a) is a special case of \cite[Proposition 7.1]{LM25}, since $\Fr$ is a commutative algebra, and the quadratic relation splits as in \cref{Qsplits}. Part (b) is a consequence of part (a).
\end{proof} 
\begin{expl}\label{expl:KMSaction}
\begin{enua}
\item
Suppose that $N=1$. Then, the only $f$ in the fundamental region is given by $f(i) = 1$ for all $1\leq i \leq d$. 
If additionally $d=2$, then
\[
\begin{split}
(v_1\otimes v_1) \cdot H_1  &= (v_1 \otimes v_1)\alb_1
\\
&= (v_1 \otimes v_1) \Big(
(\sqrt{q} (1\otimes t^{k}) +q )\tfrac{1}{m}\sum\nolimits_{j=1}^m (t^j\otimes t^{-j})
-\sqrt{q}(1\otimes t^{k})
\Big)
\\
&=
v_1\otimes (v_1 \sqrt{q}t^{k})
+\sum\nolimits_{j=1}^m
(v_1 \tfrac{\sqrt{q}+q}{m}t^j)\otimes(v_1 \tfrac{t^{-j}}{m}(\sqrt{q}t^{k}+q)).
\end{split}
\]
\item
Let $N\geq 2$ and let $d=2$.
Suppose that $f(1) = 1 <  f(2) = 2$, then $f$ is in the fundamental region, $f \cdot s_1$ is given by $1\mapsto 2, 2\mapsto 1$, and then
\[
(v_1\otimes v_2) \cdot H_1 = v_2 \otimes v_1.
\]
\item
Let $N=d=2$.
If $f$ is not in the fundamental region, say,
$f(1) = 1 < f(2) = 3$, then we cannot apply \eqref{eq:fundaction} directly, i.e., $(v_1\otimes v_3) \cdot H_1 \neq v_3 \otimes v_1$.
Note that $v_1\otimes v_3 = (v_1\otimes v_1) x_2$, so the $H_1$-action can be obtained using the wreath relation
\[
x_2 H_1 = H_1 x_1 - S_1 x_1 \partial_1(x_1) = H_1 x_1 - (q-1)e_1 x_1.
\]
Therefore,
\[
\begin{split}
(v_1\otimes v_3) \cdot H_1 &= (v_1 \otimes v_1)(H_1 x_1 - (q-1)e_1 x_1)
\\
&= (v_1 \otimes v_1)(\alb_1 x_1 - (q-1)e_1 x_1)
\\
&= (v_1 \otimes v_1)((\sqrt{q} t^{k}_2+1)e_1x_1 - \sqrt{q}t^{k}_2x_1)
\\
&= (v_3 \otimes v_1)((\sqrt{q} t^{k}_2+1)e_1 - \sqrt{q}t^{k}_2).
\end{split}
\]
\end{enua}
\end{expl}

\subsection{Wreath modules}\label{sub:wreathmod}
In this section, we realize the analogs of the Kashiwara-Miwa-Stern tensor space for $\bH = \Fr[x^{\pm1}]\wr\cH(d)$ as wreath modules, following \cite[Section 6.4]{LNX25}.
For each $1\leq i \leq N$, consider the cyclic $\Fr[x^{\pm 1}]$-submodule:
\[
M_i := v_i \Fr [x^{\pm 1}]  = \bigoplus\nolimits_{j \in\bbZ} v_{i+jN} \Fr \subseteq \bbV.
\]
Hence, there is a decomposition $\bbV = M_1 \oplus \dots \oplus M_N$.
While the $M_i$'s are all isomorphic to $\bbB = \Fr[x^{\pm1}]$ as $\bbB$-modules,
it is useful to distinguish these submodules.
In particular, $M_i^{\otimes d}$ has a $\CC$-basis
\begin{equation}\label{eq:MidbasisB}
\{v^+(t^{a_1}x^{\lambda_1}\otimes \dots \otimes t^{a_d}x^{\lambda_d}) ~|~ a_i \in [m], \lambda_i \in \ZZ \},
\end{equation}
where $v^+:= (v_i^{\otimes d})$.
It is proved in \cite{LNX25} that each $\bbC$-vector space $M_i \wr N := M_i^{\otimes d} \otimes N$ has a $\bbB \wr \cH(d)$-module structure provided
that $N $ is a Specht module  over the  {\em twisted Hecke algebra} $\tcH = \tcH^{M_i}_d$ (see \cite[Section 3]{LNX25}), which depends on $M_i$ in general.

In our case, a direct calculation \cite[(6.4.2)]{LNX25}  shows that $\tcH^{M_i}_d$ is isomorphic to the subalgebra $\Fr \wr \cH(d) \subsetneq \bH$ for any $M_i$. Write $\bH_Y:= \Fr \wr \cH(d)$ for short.

While the general theory for Specht modules are given in \cite[Section 3.3]{LNX25},
here we only need the trivial and the sign modules for $\bH_Y$.
We call the  module $S_Y^{(d)} = \Fr^{\otimes d}$ on which each $H_i$ acts by $\alb_i$ the {\em trivial module}.
Similarly, the module $S_Y^{(1^d)} = \Fr^{\otimes d}$ on which each $H_i$ acts by $-\al_i$ is called the {\em sign module}. 
We sometimes use the shorthand $\triv:= S_Y^{(d)}$ and $\sgn := S_Y^{(1^d)}$ when the dependence on $d$ is unambiguous.
It is understood that $M_i \wr S_Y^{(1)} = M_i$ is a $\bbB$-module.

For the sake of completeness, we provide the precise formula for the right $\bH$-action on $M_i \wr \triv$ and on $M_i \wr \sgn$:
\begin{equation}\label{def:WMaction}
\begin{split}
(v^+ b \otimes 1)\cdot b'
:= (v^+ bb') \otimes 1,
\quad
(v^+ b \otimes 1)\cdot H_j 
&:= 
v^+ \sigma_j(b) \otimes (1 \cdot H_j) + v^+ \rho_j(b) \otimes 1
\\
&
=\begin{cases}
v^+ (\rho_j(b) +\sigma_j(b)\alb_j)\otimes 1&\textup{for }M_i \wr\triv;
\\
v^+ (\rho_j(b) -\sigma_j(b)\al_j)\otimes 1 &\textup{for }M_i \wr\sgn,
\end{cases}
\end{split}
\end{equation}
where $b, b'\in \bbB^{\otimes d}$, and $1\leq j < d$.
\begin{prop}\label{prop:MlambdaB}
\begin{enua}
\item Any wreath module of the form $M_i \wr S_Y^{(r)}$, for $1\leq r \leq d$, is a module over the subalgebra $\bbB \wr \cH(r) \subseteq \bbB \wr \cH(d)$ generated by $\bbB^{\otimes r}$ and $H_1, \dots, H_{r-1}$.
\item For any composition $\lambda = (\lambda_1, \dots, \lambda_n) \vDash d$, denote the corresponding parabolic subalgebra by 
$\bbB\wr \cH(\lambda) := (\bbB\wr \cH(\lambda_1)) \otimes \dots \otimes (\bbB\wr \cH(\lambda_r)) \subseteq \bbB\wr\cH(d)$. 
Then, as $\bH$-modules:
\[
M^\lambda \cong \Ind_{\bbB \wr \cH(\lambda)}^{\bbB \wr \cH(d)}\big( 
(M_1\wr S_Y^{(\lambda_1)})\boxtimes \dots \boxtimes (M_n \wr S_Y^{(\lambda_n)}) 
\big)
= 
\big(
\boxtimes_{i=1}^n (M_i \wr S_Y^{(\lambda_i)})
\big) \otimes_{\bbB \wr \cH(\lambda)} \bH.
\]
\end{enua}
\end{prop}
\begin{proof}
This is a special case of  \cite[Corollary 6.4.1]{LNX25}.
\end{proof}

\begin{expl}
Suppose that $N = 2 = d$. Then, $\bbV \otimes \bbV \cong M^{(2,0)} \oplus M^{(1,1)} \oplus M^{(0,2)}$.
In particular,
\[
M^{(2,0)} = (v_1 \otimes v_1) \HGIp \cong M_1 \wr S_Y^{(2)},
\]
and the precise $\HGIp$-action is given by
\[
\begin{split}
    (v_{1+2i} \otimes v_{1+2j}) \cdot H_1 
    &= (v_1 \otimes v_1) \cdot (x^i \otimes x^j) H_1
    \\
    &= (v_1 \otimes v_1) \cdot (H_1(x^j \otimes x^i) + \rho(x^i\otimes x^j))
    \\
    &= (v_1 \otimes v_1) \cdot \alb (x^j \otimes x^i)
    + (v_1 \otimes v_1)\partial(x^i\otimes x^j) (q-1)e x_1.       
    \\
    &= (v_{1+2j} \otimes v_{1+2i})\alb  
    + (v_3 \otimes v_1)(q-1)\partial(x^i\otimes x^j) e .    
    \end{split}
\]
Similarly, 
$M^{(0,2)} = (v_2 \otimes v_2) \HGIp \cong M_2 \wr S_Y^{(2)}$.
Finally,
\[
M^{(1,1)} = (v_1 \otimes v_2) \HGIp 
\cong \Ind_{\bbB\otimes \bbB}^{\bbB \wr \cH(2)} (M_1 \boxtimes M_2) =  (M_1 \boxtimes M_2) \otimes_{\bbB \otimes \bbB} (\bbB \wr \cH(2)).
\]

\end{expl}

\subsection{Spherical and antispherical modules}\label{sec:PQWPsph}
For any Specht module $S_Y^\nu \in \bH_Y$-mod, the induced module 
\[
\Ind_{\bH_Y}^{\bH}(S_Y^\nu) = S_Y^\nu \otimes_{\bH_Y} \bH
\]
is a right module over $\bH$.
In particular, when $\nu = (d)$ or $(1^d)$, we call these modules the spherical module  and the antispherical modules, respectively:
\[
\sph := \Ind_{\bH_Y}^{\bH}(S_Y^{(d)}) = \triv \otimes_{\bH_Y} \bH,
\qquad
\asph:= \Ind_{\bH_Y}^{\bH}(S_Y^{(1^d)}) = \sgn \otimes_{\bH_Y} \bH.
\]
The action of $b' \in \bbB^{\otimes d}$ is given by $(1 \otimes b)\cdot  b'
:= 1 \otimes  (bb')$ and the action of $H_j, 1\leq j \leq d-1$ is given by
\begin{equation}\label{def:Sphaction} 
(1 \otimes b)\cdot H_j := 
1 \otimes (H_j \sigma_j(b) + \rho_j(b))
=\begin{cases}
1\otimes  (\rho_j(b) +\alb_j\sigma_j(b) ) &\textup{for }\sph,
\\
1\otimes  (\rho_j(b) -\al_j\sigma_j(b) ) &\textup{for }\asph.
\end{cases}
\end{equation}
By a direction comparison between \eqref{def:WMaction} and \eqref{def:Sphaction}, we obtain the following.
\begin{prop}
As $\bH$-modules, the (anti)-spherical modules are wreath modules. To be precise,
$\sph \cong \bbB \wr \triv$ and $\asph \cong \bbB \wr \sgn$.
\end{prop}
\begin{rmk}
	The antispherical module of the affine Hecke algebra is fundamental in representation theory. 
	It is isomorphic to the Grothendieck ring of $\Rep_0$, the principal block in the category of rational representations of $\mathbf{G}(k)$ over a field $k$ of characteristic $p$ (or alternatively of Lusztig's quantum group at a $p$-th root of unity), contains combinatorial information about its structure (see \cite{LW22} and reference therein) and models the Iwahori fixed vectors in the Gelfand--Graev module of a $p$-adic group~\cite{CS18}, as well as appears in geometric equivalences involving the two sides~\cite{AB09}.   
	In this paper we essentially prove that the newly constructed antispherical module for the affine Yokonuma algebra models the pro-$p$ Iwahori fixed vectors in the Gelfand--Graev module of $p$-adic $\mathbf{GL}_d$, which leads to the question of what are other interpretations of this module. 
\end{rmk}
\subsection{Pro-$p$ Schur algebras}
Following \cite{LM25}, define the {\em wreath Schur algebra} corresponding to $\bH$ 
\begin{equation}
\bS = \bbS(N,d) := \End_{\bH }(\bbV(N)^{\otimes d}).
\end{equation}
First, we describe an explicit basis of $\bS$ in the spirit of Dipper-James' construction of the $q$-Schur algebra basis $\{e_A~|~A\in \Theta_{N,d}\}$ \cite{DJ89}.
Here,  $\Theta_{N,d}$ is the set of all $N$-by-$N$ matrices whose entries are non-negative integers summing up to $d$. 
For each $\lambda \in \Lambda_{N,d}$, let $x_\lambda = \sum_{w \in \Sigma_\lambda} T_w$.
Recall that one can then identify the permutation module $M^\lambda \in \mod$-$\cH_q(\Sigma_d)$ as $x_\lambda \cH_q(\Sigma_d)$ so that the finite Hecke algebra acts by right multiplication.
Let ${}^\lambda\Sigma^\mu$ (and $\Sigma^\mu$, resp.) be the set of shortest representatives in $\Sigma_\lambda\backslash \Sigma_d$, ($\Sigma_d/\Sigma_\mu$, resp.), and ${}^\lambda\Sigma^\mu := {}^\lambda\Sigma \cap \Sigma^\mu$ be the set of double coset representatives in $\Sigma_\lambda\backslash \Sigma_d / \Sigma_\mu$.

The following facts are well-known, see e.g. \cite{DDPW08}:
\begin{lem}\label{lem:doublecoset}
\begin{enua}
\item There is an identification:
\[
\Theta_{N,d} \cong \{ (\lambda, g, \mu) ~|~ \lambda, \mu \in \Lambda_{N,d}, \ g \in {}^\lambda\Sigma^\mu \},
\qquad
A = (a_{i,j})_{i,j} \mapsto (\lambda, g, \mu),
\]
where $\lambda_i = a_{i,1}+\dots +a_{i,N}$, $\mu_j = a_{1,j} + \dots + a_{N,j}$  count the sum of the $i$th row and column of $A$, respectively, 
and $g$ is obtained from column reading of the set-valued matrix $(\bbZ_i^\lambda \cap g \bbZ_j^\mu)_{i,j}$.
Here, $\bbZ_i^\lambda$ is the integer interval from $\lambda_1 + \dots + \lambda_i +1$ to $\lambda_1 + \dots + \lambda_{i+1}$. 
\item There is a unique strong composition $\delta = \delta(A) \in \Lambda_{n', d}$ for some $n'$ such that
$\Sigma_{\delta} = g^{-1} \Sigma_\lambda g \cap \Sigma_\mu$.
Moreover, $\delta$ is obtained by column reading of nonzero entries of $A$.
\item Write $\delta = \delta(A)$. 
Denote the longest element in $\Sigma_\lambda$ by $w_\circ^\lambda$.
Then, 
$\Sigma_\lambda g \Sigma_\mu = \{ w ~|~ g \leq w \leq w_\circ^A \}$, in which the longest element is
$w_\circ^A = w_\circ^\lambda g w_\circ'$, where $w_\circ' = w_\circ^{\delta} w_\circ^\mu$ with $\ell(w_\circ') =  \ell(w_\circ^\mu)- \ell(w_\circ^{\delta})$.
In other words, the map $\kappa:\Sigma_\lambda \times ({}^{\delta}\Sigma_\mu) \to \Sigma_\lambda g \Sigma_\mu$, $(g_1,g_2)\mapsto g_1gg_2$ is a bijection satisfying $\ell(g_1gg_2) = \ell(g_1) + \ell(g) + \ell(g_2)$.
\end{enua}
\end{lem}

\begin{expl}
If $A=\left(\begin{smallmatrix} 1&1 \\  2&0 \end{smallmatrix}\right)$, then 
$\delta(A)$ is obtained from $(a_{11}, a_{21}, a_{12}, a_{22})$ by removing the zeroes, and hence $\delta(A) = (1,2,1)$. 
The row sum and column sum vectors of $A$ are $\lambda = (2,2)$ and $\mu =(3,1)$, respectively.
Then, $A \cong ((2,2), g, (3,1))$ with $g = |1 3 4 2| = s_2 s_3$, since
$\bbZ_1^\lambda = \{1,2\}, \bbZ_2^\lambda = \{3,4\}, g\bbZ_2^\mu = \{1,3,4\}$ and $g\bbZ_2^\mu = \{2\}$.
The longest element is $w_\circ^A = (s_1s_3) (s_2s_3) (s_2) (s_2s_1 s_2) = s_1s_3 s_2s_3 s_1 s_2.$
\end{expl}
From now on, let $A \cong (\lambda, g,\mu)$, $\delta = \delta(A)$. 
Set $x^\delta_\mu := \sum\nolimits_{w\in {}^\delta\Sigma_\mu} T_w $. 
Each element in the Dipper-James basis is the right $\cH_q(\Sigma_d)$-module homomorphism determined by 
\begin{equation}
e_A: x_\mu \cH_q(\Sigma_d) \to x_\lambda \cH_q(\Sigma_d),
\qquad
e_A(x_\mu) = x_\lambda T_g x_\mu^\delta.
\end{equation}
Now, we define the following elements in $\bH$ that play the same roles as $x_\lambda, x^\delta_\mu \in \cH_q(\Sigma_d)$:
\begin{equation}\label{def:yl}
    y_\lambda := \sum\nolimits_{w\in \Sigma_\lambda} H_w \al_{w w_\circ^\lambda},
\qquad
y^\delta_\mu := \sum\nolimits_{w\in {}^\delta\Sigma_\mu} H_w \al_{w w'_\circ},
\end{equation}
where 
$\al_w = 
\al_{i_n} \sigma_{i_n}(\al_{i_{n-1}})\dots (\sigma_{i_n} \dots\sigma_{i_2})(\al_{i_1})$ is well-defined (see \cite[Lemma 5.10]{LM25}) for any reduced expression  $w = s_{i_1} \dots s_{i_n} \in \Sigma_d$.
For $\bS$, for any partially symmetric polynomial $P \in (\bbB^{\otimes d})^{\Sigma_{\delta}}$, define a map
\begin{equation}
\theta_{A,P}: y_\mu \bH \to y_\lambda \bH,
\qquad
\theta_{A,P}(y_\mu) = y_\lambda H_g P y_\mu^\delta.
\end{equation}
Thanks to \cite[Lemma 6.2]{LM25}, $y_\lambda H_i = y_\lambda \alb_i$ for any $s_i \in \Sigma_\lambda$. 
Therefore, for any $h \in \bH$, one can write
$y_\lambda h = \sum_{\eta \in {}^\lambda\Sigma} y_\lambda H_\eta b_{\eta}$ for some $b_{\eta} = b_\eta(h) \in \bbB^{\otimes d}$.
It then follows from \cite[Proposition 7.4]{LM25} that each $y_\lambda \bH$ is identified with a permutation module via
\begin{equation}\label{eq:yHM}
y_\lambda \bH \cong M^\lambda,
\qquad
y_\lambda h \mapsto \sum\nolimits_{\eta \in {}^\lambda\Sigma} H_\eta \otimes b_{\eta}(h).
\end{equation}
Therefore, $\theta_{A,P}: M^\mu \to M^\lambda$ can be regarded as an element in $\bS = \bigoplus_{\lambda, \mu \in \Lambda_{n,d}} \Hom(M^\mu, M^\lambda)$.
\begin{expl} \label{ex:yA}
Suppose $A=\left(\begin{smallmatrix} 1&1 \\  2&0 \end{smallmatrix}\right)$ with $\lambda = (2,2)$, $\mu = (3,1)$, $\delta = \delta(A) = (1,2,1)$,  $g = s_2s_3$.
Then
\[
\Sigma_\lambda = \langle s_1, s_3\rangle, 
\quad 
{}^{\delta}\Sigma_\mu = \{ g'\in \langle s_1, s_2\rangle ~|~ s_2g' > g'\} = \{1, s_1, s_1s_2\},
\]
where the longest element is $w_\circ' = s_1s_2 = w_\circ^{\delta}w_\circ^\mu$.
Since $H_g = H_2H_3$, 
$y_\mu^\delta = (H_1H_2 + H_1 \al_2 + \al_{s_1 s_2})$.
Let $f\in \Fr$, $P := f_1x_2x_3 \in (\bbB^{\otimes d})^{\Sigma_{(1,2,1)}}$. 
The most crucial step to obtain $\theta_{A,P}(y_\mu)$ is to express 
$y_\lambda H_g P y_\mu^\delta$ as a sum of elements of the form $y_\lambda H_\eta b_\eta$ for some $\eta \in {}^\lambda\Sigma$, $b_\eta \in \bbB^{\otimes d}$. 
We have
\[
{}^\lambda\Sigma = \{g'\in \Sigma_3 ~|~ s_1g' > g', s_3g' > g'\}= \{1, s_2, s_2s_1, s_2s_3, s_2s_1s_3, s_2s_3s_1s_2\}.
\]
That is, we apply the wreath relations many times to arrive at the following:
\begin{align*}
y_\lambda H_g P y_\mu^\delta
&=
y_\lambda H_2H_3 f_1x_2x_3 (H_1H_2 + H_1 \al_2 + \al_{s_1 s_2})
\\
  & = y_\lambda H_2H_3H_1H_2  x_1x_2f_3 
  + y_\lambda H_2H_3H_1  x_1f_2(x_3\al_2 - S_2x_2) 
  \\
  &+ y_\lambda H_2H_3  f_1(\sigma_1(S_{2})S_{2}x_1x_2 
+ \al_{s_1s_2}x_2x_3 
  - S_1\al_2 x_1x_3 
  - x_2\alb_2 \sigma_2(S_1)x_1),  
\end{align*}
and hence 
$\theta_{A,P}:M^\mu \to M^\lambda$ is given by
$\theta_{A,P}(1\otimes 1) = \sum_{\eta \in \{s_2s_3s_1s_2, s_2s_3s_1, s_2s_3\}} H_\eta \otimes b_\eta$, where
\[
\begin{split}
b_{s_2s_3s_1s_2} = x_1x_2 f_3,
\quad
b_{s_2s_3s_1} = x_1f_2(x_3\al_2 - S_2x_2),
\quad
\\
b_{s_2s_3} = f_1(\sigma_1(S_{2})S_{2}x_1x_2 
+ \al_{s_1s_2}x_2x_3 
  - S_1\al_2 x_1x_3 
  - x_2\alb_2 \sigma_2(S_1)x_1).
\end{split}
\]
\end{expl}

\begin{prop}\label{prop:mult}
\begin{enua}
\item
The set $\{\theta_{A,P} ~|~ A\in \Theta_{n,d}, P \in (\bbB^{\otimes d})^{\Sigma_{\delta(A)}}\}$ forms a basis of $\bS$.
\item
The multiplication can be obtained explicitly.
Suppose $A \cong (\mu, g, \nu)$ and $A' \cong (\lambda, w, \mu)$. Then,
\[
\theta_{A',P'} \theta_{A,P} = \sum\nolimits_{\eta \in {}^\lambda\Sigma} \theta_{(\lambda, \eta, \nu), P_\eta} c_\eta,
\]
where 
$P_\eta \in (\bbB^{\otimes d})^{\Sigma_{\delta(\lambda, \eta, \nu)}}$ and $c_\eta \in B^{\otimes d}$ are coefficients in the following:
\[
y_\lambda H_w P' y_\mu^{\delta(B)} H_g P y_\nu^{\delta(A)} = \sum_{\eta \in {}^\lambda\Sigma} y_\lambda H_\eta P_\eta y_\nu^{\delta(\lambda, \eta, \nu)} c_\eta.
\]
\end{enua}
\end{prop}
\begin{proof}
Part (a) is a special case of \cite[Proposition 7.5]{LM25}.
Part (b) is a direct consequence of the identification \eqref{eq:yHM}.
\end{proof}
We remark that the multiplication formulas in \cite{BLM90, DF15} for finite and affine Schur algebras are special cases of Proposition~\ref{prop:mult} when $\Fr = \bbC$, $P = P' = 1$, $w = 1$, and an assumption that $\lambda$ only differs slightly from $\mu$.
We emphasize it is impractical to expect a closed formula for an arbitrary $c_\eta$ in \cref{prop:mult}(b), as it is not available even for the finite $q$-Schur algebra.

\section{Skew polynomial quantum wreath products} \label{sec:SPQWP}
While the metaplectic Hecke algebra $\HGbI$ can be realized as a polynomial quantum wreath product since it is isomorphic to the affine Hecke algebra \cite{S04,McN12}, 
it turns out that the pro-$p$ metaplectic Hecke algebra $\HGbIp$ has a more complicated quantum wreath product structure with the base algebra being a ring of skew polynomials.
\subsection{Skew Polynomial Quantum Wreath Products}
Recall that $\Fr := \bbC[C_m] = \tfrac{\bbC[t]}{(t^m-1)}$.
Fix  $n \in \ZZ_{\geq 1}$ such that $n|m$, and let $\bar{n}:= n/\textup{gcd}(n,2)$.
Denote the corresponding ring of skew Laurent polynomials by
\[
\bbA :=  \Fr \rtimes_n \CC[x^{\pm1}],
\quad
\textup{in which}
\quad
x^{\pm1} f = \psi^{\pm1}(f) x^{\pm1},
\]
where $\psi(f) := \xi^2 f$ for some fixed primitive $n$th root of unity $\xi \in \CC$. 

For any element $\Delta = \sum_{j} f'_j \otimes f''_j \in \Fr\otimes \Fr$, define
\begin{equation}\label{S'1}
    \Delta^{(i)} := \sum\nolimits_{j} \psi^i(f'_j) \otimes f''_j \in \Fr\otimes \Fr \subseteq \bbA \otimes \bbA ,
    \quad
    \textup{for all}
    \quad i\in\mathbb{Z},
\end{equation}
and hence $\Delta = \Delta^{(n)}$.
\begin{lem}\label{lem:SxxS}
Recall that $e = \frac{1}{m}\sum_{j=1}^m t^j\otimes t^{-j}$. 
Then, $e x_2 = x_2 e^{(1)}$ and $ex_1 = x_1 e^{(-1)}$. In particular,
\[
S x_2 = x_2 S^{(1)},
\quad
S x_1 = x_1 S^{(-1)}.
\]
\end{lem}
\begin{proof}
By definition, one has
\[
e x_2 = \frac{1}{m}\sum_{j=1}^m t^j\otimes (t^{-j} x)
= \frac{1}{m}\sum_{j=1}^m t^j\otimes (x\zeta^{j} t^{-j})
=  \frac{x_2}{m}\sum_{j=1}^m (\zeta t)^j\otimes t^{-j}
= x_2 e^{(1)}.
\]
The verification for $x_1$ is similar, and hence we omit.
\end{proof}
Lemma \ref{lem:SxxS} suggests that $S_i$ does not commute with a Laurent polynomial $P \in \CC[x_1^{\pm1}, \dots, x_d^{\pm1}]$ in general, in contrast to the special case when $n=1$.

Next, let $\rho \in \End_\bbC(\bbA \otimes \bbA )$ be the map uniquely determined by
\begin{align}
\label{Leib0}
\rho(x_1) = Sx_1, 
\quad
\rho(\Delta) = 0
\quad
\textup{for all}
\quad\Delta \in \Fr\otimes \Fr,
\\
\label{Leib}
\rho(ab) = \sigma(a)\rho(b) + \rho(a)b
\quad\textup{for all}
\quad
a, b \in \bbA .
\end{align}
\begin{lem}
Let $\rho$ be the map defined in \eqref{Leib0}--\eqref{Leib}.
\begin{enua}
    \item $\rho(x_1^{-1}) = -x_2^{-1} S$.    
    \item 
Suppose that $R = q(t^k \otimes t^k)$ for $k \in \{0, m/2\}$.
Then, $\rho(x_2) = -x_1 S$ and $\rho(x_2^{-1}) = Sx_2^{-1}$. 
\end{enua}
\end{lem}
\begin{proof}
Since $\rho(x_1) = Sx_1$, the wreath relation for $b= x_1$ reads
$H x_1 = x_2 H + S x_1$.
Part (a) can be deduced by left multiplying $x_2^{-1}$ and then right multiplying $x_1^{-1}$ to this relation.

For part (b), the wreath relation can be rewritten as
\[
x_2 H = (H-S) x_1 = R H^{-1} x_1.
\]
Note first $R = \sigma(R)$ and $\rho(R) = 0$ and hence $H$ commutes with $R$.
Then, the assumption on $R$ ensures that $R x_1 = \xi^{2k} x_1 R = x_1$ because $n|2k$. 
Therefore,
\[
H x_2 = H(x_2 H)H^{-1} = H H^{-1}  x_1 (H-S) = x_1 H - x_1 S,
\]
which is, $\rho(x_2) = -x_1S$. 
Part (b) can then be proved similar to the argument in part (a).
\end{proof}
It turns out that the $\rho$ map for the skew polynomial quantum wreath product is no longer a divided difference like the Demazure operators.
Properties of $\rho$ can be summarized as below:
\begin{lem}\label{lem:SkewDemazureFacts}
\begin{enua}
\item Let $P = x_i^{a} x_{i+1}^b \in \bbC[x^{\pm1}]^{\otimes d}$. Then,
\[
\rho_i(P) = \begin{cases}
\displaystyle
(x_ix_{i+1})^b\sum_{\ell=0}^{k-1} x_{i+1}^{\ell} S_i x_i^{k-\ell}
&\textup{if }a=b+k \geq b;
\\
\displaystyle

-(x_ix_{i+1})^a\sum_{\ell=0}^{k-1} x_{i}^{k-\ell} S_i x_{i+1}^{\ell}
&\textup{if }b=a+k \geq a.
\end{cases}
\]
\item Let $f \in \Fr^{\otimes d}$ and $P \in \bbC[x_1, \dots, x_d]$. Then, 
$\rho_i( fP ) = \sigma_i(f) \rho_i(P)$. 
\end{enua}
\end{lem}
\begin{proof}
It suffices to give the proof for the case $d=2$. We drop the subscripts of $\sigma_1$ and $\rho_1$ for simplicity.
Part (a) follows from applying the $\sigma$-twisted Leibniz rule \eqref{Leib} repeatedly. For the base case of the symmetric $P$, we have
\[
\begin{split}    
\rho(x_1x_2) &= \sigma(x_1)\rho(x_2) + \rho(x_1)x_2 = x_2(-x_1S) + S x_1 x_2 = 0,
\\
\rho(x_1^{-1}x_2^{-1}) &= \sigma(x_1^{-1})\rho(x_2^{-1}) + \rho(x_1^{-1})x_2^{-1} = x_2^{-1}Sx_2^{-1} - x_2^{-1} S x_2^{-1} = 0,
\end{split}
\]
where we used Lemma \ref{lem:SxxS} to show that $Sx_1x_2 = x_1x_2 S$.
Then, assuming $a = b+k \geq b$, we have
\[
\begin{split}
    \rho(P) = \rho((x_1x_2)^bx_1^k)&= \sigma(x_1^bx_2^b) \rho(x_1^k) + \rho(x_1^bx_2^b)x_1^k
    \\
    &= (x_1^bx_2^b) (\sigma(x_1^{k-1})\rho(x_1) + \rho(x_1^{k-1})x_1)
    \\
    &= (x_1^bx_2^b) (x_2^{k-1}S x_1 
    + \sigma(x_1^{k-2})\rho(x_1)x_1 + \rho(x_1^{k-2})x_1^2)
    \\
    &= \dots = (x_1^bx_2^b) (x_2^{k-1}S x_1 
    + x_2^{k-2}S x_1^2 + \dots + Sx_1^k).
\end{split}
\]
The remaining case when $a \leq b$ can be proved similarly.
Part (b) follows from combining \cref{DemazureFacts}(a) and \eqref{Leib}.    
\end{proof}
\begin{expl}
The first few non-trivial values of $\rho(P)$ are given as follows:
\[
\begin{split}
\rho(x_1^2) 
&= Sx_1^2 + x_2 S x_1 = x_1^2S^{(-2)} + x_1x_2 S^{(-1)},
\\
\rho(x_2^2) 
&= -x_1^2S - x_1 S x_2 = -x_1^2S - x_1x_2 S^{(1)} .
\end{split}
\]
\end{expl}
From now on, we denote by $\bbA\wr\cH(d)$ the quantum wreath product determined by the following choice of parameters:
\begin{equation}\label{eq:PQWPparn}
    \sigma = \textup{flip},
\quad
S = (q-1)e,
\quad
\rho(x\otimes 1) = S (x \otimes 1),
\quad
R = q(t^k\otimes t^k),
\end{equation}
where $k \in \{0, m/2\}$. 
Here we only consider the $m/2$ case only when $m$ is even.
\begin{thm}\label{thm:SPQWPPBW}
The quantum wreath product $\bbA \wr \cH(d)$ affords the following PBW bases:
\begin{align*}
\{(t^{a_1} x^{\lambda_1} \otimes \dots \otimes t^{a_d} x^{\lambda_d} )H_w ~|~ a_i \in [m], \lambda_i \in \ZZ, w\in \Sigma_d\},
\\    
\{H_w (t^{a_1} x^{\lambda_1} \otimes \dots \otimes t^{a_d} x^{\lambda_d} ) ~|~ a_i\in [m], \lambda_i \in \ZZ, w\in \Sigma_d\}.
\end{align*}
\end{thm}
\begin{proof}
By the PBW basis theorem \cite[Theorem 3.3.1]{LNX24}, it suffices to verify the conditions \cite[(P1)--(P9)]{LM25} for our choice of $(S,R,\sigma,\rho)$. 
Note that (P1) -- (P3), (P5), (P8), (P9) follow from the facts that $\sigma$ is the flip map and that $\rho(S) = 0 = \rho(R)$. 
It suffices to check the following equalities:
\[
\tag{P4}
r_S+\rho\sigma+\sigma\rho = l_S \sigma ,
\quad
r_R+\rho^2 = l_S\rho + l_R \in \End(\bbA\otimes \bbA),
\]
\[
\tag{P6}
    \rho_1 \sigma_2 \rho_1 = r_{S_2} \sigma_2 \rho_1\sigma_2 +\rho_2\rho_1\sigma_2 + \sigma_2\rho_1\rho_2 \in \End(\bbA^{\otimes 3}),
\]
\[
\tag{P7}
    \rho_1 \rho_2 \rho_1 + r_{R_1} \sigma_1 \rho_2\sigma_1 = \rho_2 \rho_1 \rho_2 + r_{R_2} \sigma_2 \rho_1\sigma_2 \in \End(\bbA^{\otimes 3}),
\]
where $l_X$ and $r_X$ mean left and right multiplication by the element $X$, respectively.

\textit{Step 1: Verifying (P4).}
Note that the method applied in \cite[Proposition 3.10]{LM25} does not go through.
Here, we will verify (P4) by reducing it to checking the equalities on $fP$, $x_1P$ and on $x_2P$ for $f \in \Fr\otimes \Fr, P \in \bbC[x_1, x_2]$.

For $x_1P$, the first equality in (P4) follows from an induction on the degree of $P$, since
\begin{equation}    
\begin{split}
(r_S+\rho\sigma+\sigma\rho - l_S\sigma)(x_1P) 
&=
x_1PS + \rho\sigma(x_1 P) + \sigma\rho(x_1P) - Sx_2 \sigma(P) 
\\
&=
x_1PS + \rho(x_2\sigma(P)) + \sigma(x_2\rho(P)+Sx_1P) - Sx_2 \sigma(P)
\\
&=
x_1PS + x_1\rho\sigma(P) -x_1 S\sigma(P) + x_1\sigma\rho(P)
\\
&= x_1(r_S + \rho\sigma + \sigma\rho - l_S)(P).
\end{split}
\end{equation}
The second equality in (P4) also follows from an induction on the degree of $P$, since
\begin{equation}    
\begin{split}
(r_R+\rho^2-l_s\rho - l_R)(x_1P) 
&=
x_1PR + \rho(x_2\rho(P)+ Sx_1 P) - S\rho(x_1P) - R x_1 P
\\
&= x_1PR + x_1 \rho^2(P) - x_1 S\rho(P) - R x_1 P
\\
&= x_1(r_R+\rho^2-l_s\rho - l_R)(P),
\end{split}
\end{equation}
since $R$ commutes with $x_1$.
The verification for (P4) on $x_2P$ is similar and hence we omit.

For $fP$, both equalities in (P4) can be checked on polynomials $P \in \CC[x_1, x_2]$, since
\[
    \begin{split}
(r_S+\rho\sigma+\sigma\rho - l_S\sigma)(fP) 
&=f(r_S + \rho\sigma + \sigma\rho - l_S \sigma)(P) ,
\\
(r_R+\rho^2 - l_S \rho - l_R)(fP) 
&=f(r_R+\rho^2 - l_S \rho - l_R)(P),
    \end{split}
\]
In the equations above we used that $R \in Z(\Fr\otimes \Fr)$ and $S$ is weak Frobenius (see Proposition \ref{DemazureFacts}(b)).

\textit{Step 2: Verifying (P6).}
By a similar argument, the verification of (P6) and (P7) are reduced to checking the equalities on $fP$ and $x_iP$ ($1\leq i \leq 3$) for $P \in \bbC[x_1, x_2, x_3]$, which can then be proved using an induction on the degree of $P$.
Again, the verification on $x_2P, x_3P$ are similar and so we omit. 

The relevant calculation for verifying (P6) on $x_1P$ can be found in \cite[(3.2)]{LM25} (in which $\beta$ means $S x_1$ in our case).
Following the argument therein, (P6) holds on $x_1P$ as long as \cite[Lemma B.1]{LM25} holds.
Write $S_{13} := \sigma_1(S_2) = \sigma_2(S_1)$.
In our setup, it suffices to verify
$\rho_1(S_{13}x_1) = \rho_2(S_1 x_1) + S_{13}x_1 S_2$, or, equivalently,
\begin{equation}\label{newB1}   
S_2 S_1 x_1 + \rho_1(S_{13}) x_1 = \rho_2(S_1) x_1 + S_{13} x_1 S_2.
\end{equation}
Now, both $\rho_1(S_{13}) = (q-1) \sum_j \rho(1\otimes t^j) \otimes t^{-j}$ and $\rho_2(S_1) = (q-1) \sum_j t^j \otimes \rho(1\otimes t^{-j})$ are zero.
By Proposition \ref{DemazureFacts}(b), $S_2 S_1 = S_{13} S_2$, and hence \eqref{newB1} holds.

Next, (P6) holds on $fP$ since 
\begin{multline*}  
(\rho_1 \sigma_2 \rho_1 - r_{S_2} \sigma_2 \rho_1\sigma_2 -\rho_2\rho_1\sigma_2 - \sigma_2\rho_1\rho_2)(fP)
\\=
\sigma_1\sigma_2\sigma_1(f)(\rho_1 \sigma_2 \rho_1 - r_{S_2} \sigma_2 \rho_1\sigma_2 -\rho_2\rho_1\sigma_2 - \sigma_2\rho_1\rho_2)(P),
\end{multline*}
in which we used $\sigma_1\sigma_2\sigma_1(f) = \sigma_2\sigma_1\sigma_2(f)$.

\textit{Step 3: Verifying (P7).}
In order to verify (P7) on $x_1P$,
we use the exact same argument in \cite[Proposition 3.10]{LM25}, and we can deduce that (P7) holds on $x_1P$ provided (P4) holds,  \eqref{newB1}  holds, and 
\begin{equation}\label{newB2}
    S_{13} P R_1 = S_{13} P R_2.
\end{equation}
Next, since $R = q(t^k \otimes t^k)$, one obtains that
$    S_{13} R_1 = (1 \otimes t^k \otimes t^k) S_{13} 
    = S_{13} (1 \otimes t^k \otimes t^k)$.
Thus, \eqref{newB2} follows from the fact that $R_i \in Z(\bbA^{\otimes 3})$.

Finally, (P7) holds on $fP$ since 
\begin{multline*}
    (\rho_1 \rho_2 \rho_1 + r_{R_1} \sigma_1 \rho_2\sigma_1 - \rho_2 \rho_1 \rho_2 - r_{R_2} \sigma_2 \rho_1\sigma_2)(fP)
    \\
= \sigma_1\sigma_2\sigma_1(f)
(\rho_1 \rho_2 \rho_1 + r_{R_1} \sigma_1 \rho_2\sigma_1 - \rho_2 \rho_1 \rho_2 - r_{R_2} \sigma_2 \rho_1\sigma_2)(P),
\end{multline*}
in which we used $\sigma_1\sigma_2\sigma_1(f) = \sigma_2\sigma_1\sigma_2(f)$.
\end{proof}
  
\subsection{The Kashiwara--Miwa--Stern tensor space}\label{sec:KMSforSPQWP}
  Let $\bbV = \bbV(N)$ be the free $\Fr$-module with basis $\{v_i\}_{i \in \bbZ}$. 
  Thus, $\bbV$ has an $\bbA$-module structure via  
  \[
  (v_i f) \cdot x^{\pm 1} := 
   v_{i\pm N} \psi^{\mp1}(f),
   \qquad
   \textup{for all}
   \quad
   f\in \Fr,
   \ i\in\ZZ.
  \]
Next, the $\bbA^{\otimes d}$-module $\bbV^{\otimes d}$ has a basis 
$\{v_f:= v_{f(1)} \otimes \cdots \otimes v_{f(d)}| f:[d] \to \bbZ\}$, on which $\Sigma_d$ acts (from the right) by place permutation.
\begin{prop}\label{prop:KMSaction}
Recall elements $e_i, \al_i \in \Fr^{\otimes d}$ from \eqref{def:e} and \eqref{def:alphab2}.
\begin{enua}
\item
There is a unique right $\bbA  \wr \cH(d)$-module structure on $\bbV^{\otimes d}$ determined by the following:
\begin{equation}\label{eq:fundaction2}
  v_f \cdot H_i = \begin{cases}
    v_{f\cdot s_i}
    &\textup{if } f(i) < f(i+1);
    \\
     v_{f}(-\al_i) 
     &\textup{if }  f(i) = f(i+1);
    \\
    v_{f\cdot s_i} q t^{k}_i t^{k}_{i+1}
    + v_{f} (q-1) e_i 
    &\textup{if }  f(i) > f(i+1),
    \end{cases}
\end{equation}
where $f$ is in the fundamental region, which means $1\leq f(i) \leq N$ for all $i$.
\item There is a decomposition into $\bbA \wr\cH(d)$-submodules as follows:
\[
\bbV^{\otimes d} \cong \bigoplus\nolimits_{\lambda \in \Lambda_{N,d}} M^\lambda,
\]
where 
each submodule $M^\lambda$ is generated by the basis element 
$v_\lambda 
:= v_{1}^{\otimes \lambda_1} \otimes 
v_{2}^{\otimes \lambda_2} \otimes
\dots
v_{N}^{\otimes \lambda_N} \in \bbV^{\otimes d}$.
\end{enua}
\end{prop}
\begin{proof}
Note that, without changing anything, the proof of \cite[Proposition 7.1]{LM25} generalizes to any skew polynomial quantum wreath product as long as (1) $R$ is central in $\bbA$, (2) PBW bases exist, and (3) the following element in $\bbA \otimes \bbA $ is not a zero divisor:
\[
\mathcal{P} := \al(x_1-x_2) + \beta.
\]
Indeed, $R =q (t^{k}\otimes t^{k})$ is central if $k \in \{0, m/2\}$.
Next, it follows from Theorem \ref{thm:SPQWPPBW} that the PBW bases exist.
Finally, $\al = (\sqrt{q}(t^{k}\otimes 1)+1)e - \sqrt{q}(t^{k}\otimes 1)$, $\beta = (q-1)e x_1$, and hence $\mathcal{P}$ is not a zero divisor.
Therefore, part (a) follows.

Part (b) is a consequence of part (a) as long as the quadratic relation splits,
which is a consequence of Lemma \ref{Qsplits2}.
\end{proof} 
\begin{expl}
Consider the special case $N=d=2$ with $\psi(f) := \zeta^2 f$ for some primitive ${n}$th root of unity $\zeta \in \CC$.
Let $f(1) = 1$, $f(2) = 5$. 
The $H_1$-action can be obtained using the wreath relation
\[
x^2_2 H_1 
= H_1 x_1^2 - x_1^2S^{(-2)} - x_1x_2 S^{(-1)}
= H_1 x_1^2 - (q-1) (x_1^2 e^{(-2)}  + x_1x_2 e^{(-1)} ).
\]
Therefore,
\[
\begin{split}
(v_1\otimes v_5) \cdot H_1 
&= (v_1 \otimes v_1)(H_1 x_1^2 - (q-1) (x_1^2 e^{(-2)}  + x_1x_2 e^{(-1)} ))
\\
&= (v_1 \otimes v_1)\alb_1 x_1^2 
- (v_1 \otimes v_1)(q-1) (x_1^2 e^{(-2)}  + x_1x_2 e^{(-1)} ),
\end{split}
\]
where
\[
    \alb x_1^2 
    = (\sqrt{q} t_2^{k} + q) ex_1^2 - \sqrt{q} t_2^{k}x_1^2 
    = x_1^2 (\sqrt{q} t_2^{k} + q) e^{(-2)} - \sqrt{q} x_1^2  t_2^{k}, 
\]
and hence
\[
\begin{split}
(v_1\otimes v_5) \cdot H_1 
&= (v_5 \otimes v_1)( (\sqrt{q} t_2^{k} + q) e^{(-2)} -  \sqrt{q} t_2^{k})
- (v_5 \otimes v_1)(q-1)  e^{(-2)}  
- (v_3 \otimes v_3)(q-1)  e^{(-1)}  
\\
&=(v_5 \otimes v_1)( (\sqrt{q} t_2^{k} + 1) e^{(-2)} -  \sqrt{q} t_2^{k}) 
- (v_3 \otimes v_3)(q-1)  e^{(-1)}.
\end{split}
\]
\end{expl}
\subsection{Wreath modules}\label{sec:SPQWPsph}
Within this section, let $N \geq 1$, 
$\bH := \bbA\wr\cH(d)$ and 
$\bH_Y$ be its subalgebra generated by $\Fr^{\otimes d}$ and $H_1, \dots, H_{d-1}$.
and 
Consider the decomposition $\bbV = M_1 \oplus \dots \oplus M_N$,
where
\[
M_i := v_i \bbA  = \bigoplus\nolimits_{j \in\bbZ} v_{i+jN} \Fr \subseteq \bbV.
\]
Then, $M_i^{\otimes d}$ has the following $\CC$-basis:
\begin{equation}\label{eq:MidbasisA}
\{
v^+
(t^{a_1}\otimes \dots \otimes t^{a_d})
(x^{\lambda_1}\otimes \dots \otimes x^{\lambda_d}) 
~|~ a_i \in [m], \lambda_i \in \ZZ \},
\end{equation}
where $v^+:= (v_i^{\otimes d})$.
A direct calculation shows that the twisted Hecke algebra $\tcH = \tcH^{M_i}$ is isomorphic to the subalgebra $\bH_Y$.

The proof of \cite[Corollary 6.4.1]{LNX25} can be repeated  verbatim to prove the following:
\begin{prop}\label{prop:MlambdaA}
\begin{enua}
\item Any wreath module of the form $M_i \wr S_Y^{(r)}$, for $1\leq r \leq d$, is a module over the subalgebra $\bbB \wr \cH(r) \subseteq \bbA \wr \cH(d)$ generated by $\bbA^{\otimes r}$ and $H_1, \dots, H_{r-1}$.

\item For any composition $\lambda = (\lambda_1, \dots, \lambda_n) \vDash d$, denote the corresponding parabolic subalgebra by 
$\bbA\wr \cH(\lambda) := (\bbA\wr \cH(\lambda_1)) \otimes \dots \otimes (\bbA\wr \cH(\lambda_r)) \subseteq \bbA\wr\cH(d)$. 
Then, as $\bH$-modules:
\[
M^\lambda \cong \Ind_{\bbA \wr \cH(\lambda)}^{\bbA \wr \cH(d)}\big( 
(M_1\wr S_Y^{(\lambda_1)})\boxtimes \dots \boxtimes (M_n \wr S_Y^{(\lambda_n)}) 
\big)
= 
\big(
\boxtimes_{i=1}^n (M_i \wr S_Y^{(\lambda_i)})
\big) \otimes_{\bbA \wr \cH(\lambda)} \bH.
\]
\end{enua}
\end{prop}
By an argument similar to the one in Section \ref{sec:PQWPsph}, we call the following modules the spherical module and the antispherical module, respectively:
\[
\sph := \bbA \wr \triv \cong \triv \otimes_{\bH_Y} \bH,
\qquad
\asph:= \bbA \wr \sgn \cong  \sgn \otimes_{\bH_Y} \bH.
\]

\subsection{Wreath Schur algebras} 
Define the wreath Schur algebra corresponding to $\bH = \bbA \wr \cH(d)$ as
\[
\bbS(N,d) := \End_{\bH }(\bbV(N)^{\otimes d}).
\]
Define the elements $y_\lambda, y_\lambda^\delta \in \bH$ as in \eqref{def:yl}.
For any partially symmetric polynomial 
$P \in (\bbA ^{\otimes d})^{\Sigma_{\delta}}$, define a map
\begin{equation}
\theta_{A,P}: y_\mu \bH \to y_\lambda \bH,
\qquad
\theta_{A,P}(y_\mu) = y_\lambda H_g P y_\mu^\delta.
\end{equation}
Again, for any $h \in \bH$, one can write
$y_\lambda h = \sum_{\eta \in {}^\lambda\Sigma} y_\lambda H_\eta b_{\eta}$ for some $b_{\eta} = b_\eta(h) \in \bbA^{\otimes d}$.
The proof of \cite[Proposition 7.4]{LM25} generalizes, without changing a word, to a proof of the parallel statement in our case when the base algebra is a ring of skew polynomials.
Hence, each $y_\lambda \bH$ is identified with a permutation module via
\begin{equation}\label{eq:yHM2}
y_\lambda \bH \cong M^\lambda,
\qquad
y_\lambda h \mapsto \sum\nolimits_{\eta \in {}^\lambda\Sigma} H_\eta \otimes b_{\eta}(h).
\end{equation}
Thus, $\theta_{A,P}: M^\mu \to M^\lambda$ can be regarded as an element in $\bbS(N,d) = \bigoplus_{\lambda, \mu \in \Lambda_{N,d}} \Hom(M^\mu, M^\lambda)$.

\begin{prop}\label{prop:mult2}
\begin{enua}
\item
The set $\{\theta_{A,P} ~|~ A\in \Theta_{n,d}, P \in (\bbA^{\otimes d})^{\Sigma_{\delta(A)}}\}$ forms a basis of $\bbS(N,d)$.
\item
The multiplication can be obtained explicitly.
Suppose $A \cong (\mu, g, \nu)$ and $A' \cong (\lambda, w, \mu)$. Then,
\[
\theta_{A',P'} \theta_{A,P} = \sum\nolimits_{\eta \in {}^\lambda\Sigma} \theta_{(\lambda, \eta, \nu), P_\eta} c_\eta,
\]
where $P_\eta \in (\bbA^{\otimes d})^{\Sigma_{\delta(\lambda, \eta, \nu)}}$ and $c_\eta \in \bbA^{\otimes d}$ are coefficients appearing in the following:
\[
y_\lambda H_w P' y_\mu^{\delta(A')} H_g P y_\nu^{\delta(A)} = \sum_{\eta \in {}^\lambda\Sigma} y_\lambda H_\eta P_\eta y_\nu^{\delta(\lambda, \eta, \nu)} c_\eta.
\]
\end{enua}
\end{prop}
\begin{proof}
Part (a) can be proved using exactly the same argument as in \cite[Lemma 6.2, Propositions 6.3 and 7.4--5]{LM25}, provided the corresponding results hold for the ring of skew polynomials.
Firstly, the PBW bases theorem (Theorem \ref{thm:SPQWPPBW}) holds in our case.
Secondly, the quadratic relation splits in both ways as in Lemma \ref{Qsplits2}.

The argument in {\em loc. cit.} then proves that the set 
$\{ y_\mu P H_g y_\lambda^{\delta(\lambda,g,\mu)} ~|~ g \in {}^\lambda\Sigma^{\mu}, P \in \mathbb{B}_g\}$ forms a $\CC$-basis of $y_\lambda \bH \cap \bH y_\mu$, where $\mathbb{B}_g$ is a $\CC$-basis of $(\bbA^{\otimes d})^{\Sigma_{\delta(\lambda,g,\mu)}}$,
and hence the described set is indeed a basis of $\bbS(N,d)$.

Part (b) is a direct consequence of the identification \eqref{eq:yHM2}.
\end{proof}
\section{Hecke Algebras on $p$-adic Groups and Their  Modules}\label{sec:padic}
In this section, we present results of~\cite{GGK1} on presentations of (metaplectic) pro-$p$ Iwahori Hecke algebras and their Gelfand--Graev modules and define the metaplectic pro-$p$ Schur algebra. 
\subsection{Preliminaries on $p$-adic groups}
Let $F$ be a local non-Archimedean field with ring of integers $\cO$, maximal ideal $\varpi \cO$ for fixed uniformizer $\varpi$ and residue field $\kappa := \cO/ \varpi \cO$ of order $q$ a power of a prime number. 
For example one may pick $F=\QQ_p$, $\cO = \ZZ_p$, $\varpi = p \in \ZZ_p$ and $q=p\in \CC$.
 
Let $G := \mathbf{GL}_d(F)$ with split torus $T = \mathbf{T}(F)$ consisting of diagonal matrices, positive Borel $B$ consisting of upper triangular matrices, negative Borel $B^-$ consisting of lower triangular matrices and corresponding positive and negative unipotent subgroups $U \subset B$, $U^- \subset B^-$.
We denote by $Y:= \Hom (\mathbf{G}_m, \mathbf{T})$ and $X:= \Hom (\mathbf{T}, \mathbf{G}_m)$ the cocharacter and character lattice of $G$, respectively. 
We may identify $Y \cong \ZZ^d $ and $X \cong \ZZ^d$ in such a way that the canonical pairing $\langle -,-\rangle :X \times Y \to \ZZ$ is identified with the standard inner product on $\ZZ^d$. 
If we set $\{\epsilon_i\}_{1 \leq i \leq d}$ to be the standard $\ZZ$-basis of $\ZZ^d$, we have $Y = \bigoplus_{i=1}^d \mathbb{Z} \epsilon_i$.
We write an element $\lambda = \sum_{i=1}^d \lambda_i \epsilon_i \in Y$ as $\lambda = (\lambda_1, \cdots, \lambda_d)$ with $\lambda_i \in \ZZ$, and 
we say that $\lambda$ is dominant if $\lambda_i \geq \lambda_{i+1}$.
Denote by $\Phi = \Phi_+ \sqcup \Phi_- \subset X$ the partition of the set of roots into positive and negative roots (induced by the choice of Borel) and by $\Phi^\vee = \Phi^\vee_+ \sqcup \Phi^\vee_- \subset Y$ the corresponding partition on the set of coroots.
The set of simple roots $\sroots$ consists of $\{\sroot_i =\epsilon_i -\epsilon_{i+1}\}_{1\leq i \leq d-1}$. 

Let $W$ be the finite Weyl group $W := N(\mathbf{T})/\mathbf{T} \cong \Sigma_d$, the symmetric group on $d$ letters.  
It has simple reflections $s_{\sroot_i} = s_i\in W$ for $1 \leq i \leq d-1$. 
It acts naturally on elements $\mu \in X$ and $\lambda \in Y$ (once you identify either $X$ or $Y$ with $\ZZ^d$, the action is just permutation of components).

Let $\Waff:=W\ltimes Y \cong \bbZ \wr \Sigma_d$ be the extended affine Weyl group.

\subsection{Metaplectic covers of $p$-adic groups}
\label{sub:metaplecticgroups}
Fix a positive integer $n$ and let $\mu_n$ denote the group of $n$-th roots of unity.
Assume $F$ contains $n$ distinct $n$-th roots of unity (equivalent to saying $q-1 \equiv 0 \pmod{n}$). 
Let $(\cdot, \cdot)_n:F^* \times F^* \to \mu_n$ denote the $n$-Hilbert symbol as defined in~\cite[p.11]{GGK1} or~\cite[\S2.2.2]{BP25}.
Denote $\epsilon := (\varpi, \varpi)_n=(-1,\varpi)_n \in \{\pm1\}$.
We will impose the additional condition $q-1 \equiv 0 \pmod{2n}$ which implies $\epsilon=1$ for our main results in \S\ref{sec:idemt}; this condition appears quite often in the literature (ex: \cite{BBB19,PP17}), though we do note that with some tedious computations one can adapt the setup in previous sections to take this extra complication into account. 
For $\chi \in \Hom (\kappa^*,\CC^*)$, $\psi \in \Hom (\kappa, \CC^*)$ define the Gauss sums:
\begin{equation}\label{eq:def:gausssum}
	\bg(\psi, \chi):= \sum_{u\in \kappa^*} \psi(u)\chi(u).
\end{equation}

Let $\Bs:Y \times Y \to \ZZ$ be a Weyl group invariant bilinear form on $Y$.
We denote by $\Qs:Y \to \ZZ$ the associated quadratic form satisfying $\Bs(y,z):=\Qs(y+z)-\Qs(y)-\Qs(z)$. 
To the data $(\Bs, n)$ (or $(\Qs, n)$) we may associate a Brylinski--Deligne central extension $\mG$ of $G$~\cite{BD01}(see also~\cite{GG18} or~\cite[\S2.3]{GGK1} for details on this construction):
\[ 1 \to \mu_n \to \mG \to G \to 1.\] 
Since $Y \cong \bigoplus_{i=1}^d \bbZ \epsilon_i$,
a bilinear form with the properties above will depend on two integers $\mathbf{p}, \mathbf{q}\in \ZZ$:
\begin{equation}
	\Bs(\epsilon_i, \epsilon_j) = 
	\begin{cases}
		2\mathbf{p} & \text{ if } i=j; \\
		\mathbf{q} & \text{ if } i\neq j.
	\end{cases}
\end{equation}
We will focus on the case when $\mathbf{q}$ is divisible by $n$. 
For simplicity, we set $\mathbf{p} = 1, \mathbf{q}=0$; the general case can be reduced to this case quite easily.
This implies $\Qs (\alpha_i)=1$ for $\alpha_i \in \Delta$. 
This cover is sometimes called Savin's nice cover and is the one appearing in~\cite{GGK2} or~\cite{BBB19, BBBG}. 
An alternative cover, with $2\mathbf{p}-\mathbf{q}=-1$ is called the Kazhdan--Patterson cover and is quite important in the theory of automorphic forms.
It seems non-trivial at the moment to extend our results to the Kazhdan--Patterson covers or to the most general cover of $G$. 

Let us define $\mY:= \{ y \in Y, \Bs(y, z)\cong 0 \pmod{n} \mbox{ for all } z\in Y\}$.
In our case, we have $\mY = \barn Y$, where $\barn = \frac{n}{\gcd(2,n)}$.

\subsection{Metaplectic Hecke algebras}

We would like to define Hecke algebras with respect to certain compact subgroups of $G$. 
Let $K = \mathbf{G}(\cO)$ be the maximal compact subgroup of $G$. 
The reduction mod $\varpi$ map $\redp: \cO \to \kappa := \cO / \varpi \cO$ induces a map $\redp : K \to \mathbf{G}(\kappa)$ and we define the Iwahori subgroup $I:= \redp^{-1}(\mathbf{B}(\kappa))$ and the pro-p Iwahori subgroup $\Ip:= \redp^{-1}(\mathbf{U}(\kappa))$. 

The group $\mG$ splits (canonically) over $U^-$ and (non-canonically) over $K$; we fix a splitting $s_K:K \to \mG$.  
We can therefore view $U^-$  and $\Ip \subset I \subset K$ as subgroups of $\mG$.
This allows us to define metaplectic Hecke algebras (and their Gelfand--Graev modules). 
Fix an embedding $\iota:\mu_n \to \CC^*$.
Define the metaplectic pro-$p$ Iwahori Hecke algebra of $\iota$-\emph{genuine} functions: 
\begin{equation}\label{def:HGmI}
	\HGbIp := C^\infty_{c,\iota}(\Ip \backslash {\mG} / \Ip) = \{f \in C_c^\infty(\mG) | f(\zeta k_1 g k_2) = \iota(\zeta)f(g) \mbox{ for all } g \in \mG, k_1, k_2 \in \Ip, \zeta \in \mu_n\}.
\end{equation}
$\HGbIp$ is an algebra, with multiplication given by convolution: $f_1 * f_2 (g) = \int_{G} f_1(h)f_2(h^{-1}g)dh$, where $dh$ is the Haar measure on $G$ normalized such that $\int_{\Ip}dh=1$.  
We can also define 
\begin{equation}
\HKbIp := C^\infty_{c,\iota}(\Ip \backslash {\mK} / \Ip) \cong C^\infty_{c}(\Ip \backslash {K} / \Ip);
\end{equation}
the isomorphism above is due to the splitting of $K$ and the genuine condition on the left side.
Similarly, one may define the algebras $\HGbI := C^\infty_{c,\iota}(I \backslash {\mG} / I)$ and $\HKbI:=  C^\infty_{c,\iota}(I \backslash {\mK} / I) \cong C^\infty_{c}(I \backslash {K} / I)$.
Note that if we set $e_I \in \HGbIp$ to be the characteristic function of $I\in \mG$, then one has the standard $p$-adic isomorphism
\begin{equation}\label{eq:proptoIwahori}
	e_I \HGbIp e_I \cong \HGbI.
\end{equation}

Denote by $\cT_g$ the element in the corresponding Hecke algebra ($\HGbIp$ or $\HKbIp$) that is supported on $\mu_n\Ip g\Ip$ and which takes value $1$ at $g$.
This element is well-defined cf.~\cite[\S3.1]{GGK1}. 

For $\alpha \in \Phi$, let $\{\bar{e}_\alpha (a), a \in F\}$ be the corresponding root subgroups in $\mG$ (cf.~\cite[\S2.3]{GGK1}). Define 
\begin{equation}
	\bar{w}_\sroot(a) := \bar{e}_{\sroot}(a) \bar{e}_{-\sroot}(-a^{-1}) \bar{e}_{\sroot}(a), \quad \bar{h}_{\sroot}(a):= \bar{w}_{\sroot}(a) \bar{w}_{\sroot}(-1) \mbox{ for } a \in F, \quad
	\cT_{\sroot} := \cT_{\bar{w}_\sroot(1)}.
\end{equation}

\subsubsection{The finite Yokonuma algebra}\label{subsubsec:finiteYokonuma}
Recall the finite field $\kappa$ with $q$ elements. 
The group $\kappa^*$ is cyclic of order $q-1$ with generator $t$, so we may write $\kappa^* \cong C_{q-1}$. 

Let us denote by $\Gk, \Bk, \Uk, etc.$ the $\kappa$ points of the corresponding group (i.e. $\Gk := \mathbf{G}(\kappa)$, etc.). 
Then $\Tk = \mathbf{T}(\kappa) \cong C_{q-1}^d$. 
We have a natural algebra isomorphism $C^\infty_{c}(\Uk \backslash {\Gk} / \Uk) \cong C^\infty_{c}(\Ip \backslash {K} / \Ip) = \HKbIp$ induced by the map $\redp$. 
Both algebras have basis indexed by $W \ltimes \Tk \cong \Uk \backslash {\Gk} / \Uk \cong \Ip \backslash {K} / \Ip$.

For $\chi \in \Hom (T_\kappa, \CC^*) = \Hom (T_\kappa, \mu_{q-1})$ define:
\begin{equation}
	c(\chi):= \frac{1}{|T_\kappa|} \sum_{\tau \in T_\kappa}\chi(t)\cT_\tau \in C^\infty_{c}(\Uk \backslash {\Gk} / \Uk) \cong \HKbIp.
\end{equation}
The element above can also be thought of as an element in $\HGbIp$; we shall denote both elements by the same symbol; it should be clear from context where any given elements lives. 
Define:
\begin{equation}\label{eq:def:calphai}
	c_\sroot(\chi)= \frac{1}{q-1} \sum_{u \in \kappa^*}\chi(u)\cT_{\bar{h}_{\sroot}(u)} \textup{ for } \chi \in \Hom(\kappa^*, \mu_n).
\end{equation}
The algebra $\HKbIp$ is generated by $\cT_{\sroot_i}, \sroot_i\in \sroots$ and $c(\chi)$, $\chi\in \Hom(T_\kappa, \CC^*)$ subject to relations:
\begin{align}
	& \label{eq:pQR} \cT_{\sroot_i} c(\chi) = c(s_i \chi)\cT_{\sroot_i},
	\\
	& \label{eq:pQR} \cT_{\sroot_i}^2 = q \cT_{\bar{h}_{\alpha_i}(-1)} + (q-1)c_{\alpha_i}(\mathbf{1}) \cT_{\sroot_i},
	\\
	& \label{eq:pQR} c(\chi) c(\chi')=c(\chi') c(\chi).
\end{align}  
and the braid relations for the $\cT_{\sroot_i}$ (cf.~\cite[Prop. 3.4, Lemmas 3.2 and 4.5]{GGK1}).

The braid relations allows us to define $\cT_{w}:=\cT_{\sroot_{i_1}}\cdots\cT_{\sroot_{i_k}}$ for a reduced expression $w= s_{i_1},\cdots s_{i_k}$. 
Then $\HKbIp$ has basis $\{c(\chi) \cT_w\}$ for $\chi \in \Hom(T_\kappa, \CC^*), w \in W$.
 
\subsubsection{The metaplectic pro-$p$ Hecke algebra}\label{subsub:metaplecticYokonuma}
The structure of $\HGbIp$ is described in~\cite{GGK1} (and in the $n=1$ setting by Vigneras~\cite{V16}), which we will now present.
The pro-$p$ Iwahori Weyl group is 
\begin{equation}
	\Wp := (\Tk \times Y)\rtimes W = (C_{q-1}^d \times \ZZ^d)\rtimes \Sigma_d \cong (C_{q-1} \times \bbZ)\wr \Sigma_d.
\end{equation}
An arbitrary element in $\Wp$ is of the form $a\lambda w$ for some  $a\in \Tk$, $\lambda \in Y$, $w\in W$.

Let $Y_\lambda$ denote the element $\cT_{\varpi^{\lambda'}}\cT_{\varpi^{\lambda''}}^{-1}$ for any $\lambda', \lambda''$ dominant such that $\lambda= \lambda'-\lambda''$ ($Y_\lambda$ is indeed well-defined based on the same argument used in the Iwahori Hecke algebra case). 
Moreover, let 
\begin{equation}\label{def:varphis}
	\varphi(\lambda) \in \Hom (T_\kappa, \CC^*), \quad \varphi(\lambda) (s):= [\varpi^\lambda, s],
\end{equation}
where the commutation relation is in the metaplectic group. 

The following result gives a presentation by generators and relations of $\HGbIp$; it is a summary of~\cite[Lemma 3.2, Prop. 3.4, Lemma 4.5, Prop 4.6, Thm. 4.18]{GGK1}. 

\begin{thm}[\cite{GGK1}]\label{thm:propHgenrel}
	$\HGbIp$ has generators $\cT_{\alpha_i}$ for $\sroot_i\in \sroots$, $Y_\lambda$ for $\lambda\in Y$, $c(\chi)$ for $\chi\in \Hom(T_\kappa, \CC^*)$ which are 
	subject to the following relations
	\begin{align}
		& \label{eq:MpBR} \textup{(braid relations)} 
		& \textup{braid relations for the } \cT_{\sroot_i},
		\\
		& \label{eq:MpQR} \textup{(quadratic relations)}
		&
		\cT_{\sroot_i}^2 = q \cT_{\bar{h}_{\alpha_i}(-1)} + (q-1)c_{\alpha_i}(\mathbf{1}) \cT_{\sroot_i},
		\\
		&\label{eq:MpPoly} \textup{(commutation relations)} 
		& Y_{\lambda} Y_\mu = Y_{\lambda+\mu}=Y_{\mu} Y_\lambda, \quad c(\chi) c(\chi') = c(\chi') c(\chi),	
		\\
		&\label{eq:MpSkew} \textup{(mixed relations)} 
		& \cT_{\sroot_i} c(\chi) = c(s_i \chi)\cT_{\sroot_i}, \quad 
		Y_\lambda c(\chi) = c(\varphi(\lambda) \chi) Y_\lambda,
		\\
		& \textup{(Bernstein relations)}
		&\cT_{\sroot_i} Y_\lambda +
		(q-1) \sum_{j \in \ZZ_{\geq 1}} \epsilon^{j(1+\langle \lambda,\alpha_i \rangle)}Y_{\lambda+j\alpha_i} c_{\alpha_{i}}(j+\langle \lambda,\alpha_i\rangle)= \quad\quad
		\notag
		\\
		&\label{eq:MpWR}
		& Y_{s_i \lambda} \cT_{\sroot_i} +
		(q-1)\sum_{j \in \ZZ_{\geq 1-\langle \lambda, \alpha_i \rangle}} \epsilon^{j(1+\langle \lambda,\alpha_i \rangle)}Y_{\lambda+j\alpha_i} c_{\alpha_{i}}(j+\langle \lambda,\alpha_i\rangle).
	\end{align} 
\end{thm}
There is a natural bijection $\Wp \cong \mu_n \Ip \backslash \mG / \Ip$~\cite[Prop. 3.35]{V16}; as such $\HGbIp$ has a basis indexed by elements in $\Wp$.
The relations in Theorem~\ref{thm:propHgenrel} allow us to specify this basis.  
\begin{cor}\label{cor:propHbasis}
	$\HGbIp$ has basis $\{c(\chi) Y_\lambda \cT_w \}$ (or basis $\{Y_\lambda c(\chi) \cT_w \}$) for $\chi \in \Hom(T_\kappa, \CC^*), \lambda \in Y, w\in W$.
\end{cor}
It is a standard fact that the map $c(\chi) \cT_w \mapsto c(\chi)\cT_w$ realizes $\HKbIp$ as a subalgebra of $\HGbIp$.
 
\subsection{Metaplectic Gelfand--Graev modules}

Let $\psi: U^-\to \CC^\times$ be a nondegenerate character of conductor $\varpi \cO$ as in~\cite[p. 30]{GGK1}. This character descends to a non-degenerate character (denoted by the same symbol; it should be clear which character we use at any given time) $\psi: \Uk^- \to \CC^*$.   

\subsubsection{Finite Gelfand--Graev module}\label{subsub:finiteGG}
Let $\Vk := \Ind_{\Uk^{-}}^{\Gk} (\psi)^{\Uk}$. 
The algebra $\HGbIpk := C_c^{\infty} (\Ip \backslash K /  \Ip)$ is isomorphic to  $C_c^{\infty} (U_\kappa \backslash G_\kappa /  U_\kappa)$ and acts on $\Vk$ by convolution. 
Then $\Vk \cong \CC[\Tk]$ as $\HGbIpk$-modules~\cite[Lem 5.4]{GGK1}.   

More precisely, $\Vk$ has basis $c^\bg(\chi)$ for $\chi\in \Hom(T_\kappa, \CC^*)$. 
The action of $c(\chi) \in \HGbIpk$ on $\Vk$ is given by multiplication in $\HGbIpk$. 
The action of $\cT_{\sroot_i}$ is given by~\cite[(5.2)]{GGK1}:
\begin{equation}\label{eq:action-finite-Vk}
	c^\bg(\chi) \cT_{\sroot_i} = \bg(\psi_{\sroot_i}, \chi_{\sroot_i}) c^\bg(s_i \chi),
\end{equation}
where $\chi_{\sroot}(u):=\chi(\hbar_{-\sroot}(u))$ and $\psi_{\sroot}:= \psi(\bar{e}_{-\sroot}(u))$ and $\bg$ is the Gauss sum defined in~\eqref{eq:def:gausssum}. 

To specify the Gauss sums, note that once you identify $\chi_{\alpha_i}$ with an element $u^k\in\kappa^*$, then $\bg(\psi_{\sroot_i}, \chi_{\sroot_i}) = \mathbf{g}(\psi_\alpha, (u,\varpi)_n^k)$ (cf.~\cite[p. 65]{GGK1}).
The character $\psi$ is a fixed non-degenerate additive character with the properties discussed above, the Gauss sum appearing in~\eqref{eq:action-finite-Vk} do not depend on it or on the simple root $\alpha_i$.
Because of this, if we identify $\chi$ with $\gamma\in[m]^d$ (as we will do in~\S\ref{subsec:toral}), then $\bg_{\chi_{\sroot_i}}$ only depends on $k_{\chi_{\sroot_i}}=\gamma_i-\gamma_{i+1}$, so we may write $\bg (\psi_{\sroot_i}, \chi_{\sroot_i})=\bg_{k_{\chi_{\sroot_i}}}$ with $k_{\chi_{\sroot_i}}\in \ZZ$. 
These Gauss sums will satisfy the relations (cf.~\cite[p. 65]{GGK1},~\cite[p. 33]{GSS23}, see also~\cite[(4.1)]{PP17}):
\begin{equation}\label{eq:gauss-sums-properties}
	\bg_k = \bg_{k+m}, 
    \quad 
    \bg_k \bg_{-k}=q \quad\textup{ for } k\neq 0 \pmod{m},
    \quad \textup{ and } 
    \quad \bg_0=-1.
\end{equation}
While it is true that $|\bg_k|=\sqrt{q}$ when $k \neq 0 \pmod{m}$, $\bg_k$ is generally not equal to $\pm \sqrt{q}$. 
Gauss sums contain important number theoretic information, so it is important to keep accurate track of them.

\begin{rmk}
	We note here that the appearance of the Gauss sums in~\eqref{eq:action-finite-Vk} is a pro-$p$ Iwahori phenomena and \emph{not} a metaplectic phenomena. 
	These Gauss sums appear non-trivially even when the degree of the metaplectic cover we are working with is $n=1$, while when working at the Iwahori level we only notice Gauss sums when working with metaplectic covers.
\end{rmk}
\subsubsection{The $p$-adic Gelfand--Graev module}
Define the $\HGbIp$ Gelfand--Graev module 
\begin{equation}\label{eq:defWIp}
	\WbIp := \ind_{\mu_n U^-}^{\mG} (\iota \otimes \psi)^{\Ip}.  
\end{equation}
Similarly we may define the $\HGbI$ module $\WbI := \ind_{\mu_n U^-}^{\mG} (\iota \otimes \psi)^{I}$. As in~\eqref{eq:proptoIwahori}, we have that
\begin{equation}\label{eq:proptoIwhaoriGG}
	\WbIp e_I \cong \WbI.
\end{equation}
The following description of $\WbIp$ is given in~\cite[Thm 5.13]{GGK1}:
\begin{prop}\label{prop:affineWIpinduced}
	There is an $\HGbIp$ isomorphism
	\begin{equation}\label{eq:affineWIpinduced}
		\WbIp \to \Vk \otimes_{\HGbIpk} \HGbIp.
	\end{equation}
\end{prop} 
We know that $\HGbIpk$ and $\HGbIp$ have bases $\{c(\chi) \cT_w \}$ and $\{Y_\lambda c(\chi) \cT_w \}$, respectively, cf.~\S\ref{subsubsec:finiteYokonuma} and Cor. \ref{cor:propHbasis}.
The presentation of $\Vk$ above implies that $\WbIp$ has basis $\{c^\bg(\chi) Y_\lambda\}$ for $\lambda \in Y, \chi \in  \Hom(T_\kappa, \CC^*)$.	

Let us define the $p$-adic Schur algebra
\begin{defn}
	The Schur algebra $\SGbIp$ corresponding to the metaplectic pro-$p$ Iwahori Hecke algebra $\HGbIp$ and its Gelfand--Graev module $\WbIp$ is
	\begin{equation}\label{eq:def:padicSchur}
		\SGbIp := \End_{\HGbIp} (\WbIp).
	\end{equation}
\end{defn}

\section{The Identification Theorems}\label{sec:idemt}
In this section we prove our main results at the pro-$p$ level. 
We identify the pro-$p$ metaplectic Hecke algebra $\HGbIp$ with the skew polynomial quantum wreath product $\bbA \wr \cH(d)$, as well as identify the Gelfand--Graev module $\WIp$ with the (generalized) antispherical module $\bbV(1) \wr \sgn$.
As corollary of these two results and Proposition~\ref{prop:mult2}, we give a basis with multiplication formula for the $p$-adic Schur algebra $\SGbIp$.
Throughout, we explain how the main results at the pro-$p$ level can be used to derive results at the Iwahori level, finishing with a presentation of the $p$-adic Schur algebra at the Iwahori level, recovering the main result of~\cite{GGK2}.

\subsection{The toral elements $c(\chi) \in T_\kappa$}\label{subsec:toral}
Set $m = q-1$ to be the order of the cyclic group $\kappa^\times \cong C_m$ generated by an element $t$. 
Then $T_\kappa \cong C_m^d$. 
Write $\Fr = \mathbb{C}[t]/(t^m-1)$.
We use the identification:
\[
T_\kappa \cong \Fr^{\otimes d},
\qquad
\tau^a := (t^{a_1}, \dots, t^{a_d})
\mapsto
t^{a_1}\otimes \dots \otimes t^{a_d},
\]
where $a = (a_i)_i \in\ZZ^d$.
Next, let $u$ be a primitive $n$-root of unity and let $\zeta = (\varpi, t)_n = u^l \in \mu_n \subseteq \CC^\times$.
Any character $\chi \in \Hom(T_\kappa, \mathbb{C}^\times)$ can be identified with $\chi = (\chi_1,\cdots,\chi_d)$, where $\chi_j\in \Hom(\kappa^*, \mathbb{C}^\times)$ is given by $\chi_j (t):=\chi(\tau^{\epsilon_j})$. 
Then for some $\gamma_j \in [m]$
we have:
\[
\chi(\tau^a) = \chi_1(t)^{a_1} \dots \chi_d(t)^{a_d} = u^{\gamma_1a_1 + \dots + \gamma_d a_d} \in \CC^\times.
\]
In other words, there is a bijection
$
[m]^d
\to
\Hom(T_\kappa, \mathbb{C}^*)$,
$\gamma=(\gamma_1, \dots, \gamma_d) 
\mapsto \chi_\gamma$.
Therefore, any $c(\chi)$ is of the following form for some $\gamma \in [m]^d$:
\[
c(\chi^\gamma) \cong \frac{1}{m^d} 
\sum_{a \in[m]^d} 
(u^{\gamma_1} t)^{a_1} \otimes \dots \otimes (u^{\gamma_d} t)^{a_d}
\in \Fr^{\otimes d}.
\]
In particular, $c(s_i \chi^{\gamma}) = c(\chi^{s_i \cdot \gamma})$ where $\Sigma_d$ acts on $[m]^d$ by place permutation.
Note that the matrix $(\frac{1}{m} u^{ij})_{i,j}$ is invertible with inverse $(u^{-ij})_{i,j}$, therefore one can define another basis $\{\tp{j} \}_{j=0}^{m-1}$ of $\Fr$:
\begin{equation}\label{tintot'}
\tp{j} :=  \tfrac{1}{m} \sum\nolimits_{i=1}^m (u^j t)^i,
\quad \textup{ such that } \quad 
t^j = \sum\nolimits_{i=1}^m u^{-ij} (\tp{i}).    
\end{equation}
We choose to use the pre-subscript for ${}_jt'\in\Fr$ instead of a subscript since, following our convention \eqref{def:subscriptshorthand}, a subscript is reserved for an element in $\Fr^{\otimes d}$.
With this basis $\{\tp{j}\}_j$, one then obtain a simpler expression for the toral elements as below: 
\begin{equation}
c(\chi^{(\gamma_1, \dots, \gamma_d)}) = (\tp{\gamma_1}) \otimes \dots \otimes (\tp{\gamma_d}) \in \Fr^{\otimes d}.
\end{equation}

\begin{expl}
Suppose that $m=2$. Then 
\[
\tp{1} = \tfrac{1+ut}{2},
\quad
\tp{0} = \tfrac{1+t}{2},
\quad
c(\chi^{(1,0)}) = \tfrac{1}{4}(1\otimes 1 +ut\otimes 1 + 1\otimes t + ut\otimes t)
= (\tp{1}) \otimes (\tp{0}).
\]
Moreover, when both pre-scripts and subscripts are in use, we have
\[
\tp{0}_1 = (\tfrac{1+t}{2})\otimes 1,
\quad
\tp{0}_2 = 1\otimes (\tfrac{1+t}{2}),
\quad
\tp{1}_1 = (\tfrac{1+ut}{2})\otimes 1,
\quad
\tp{1}_2 = 1\otimes (\tfrac{1+ut}{2}).
\]
\end{expl}
Next, for $\lambda = \sum_j \lambda_j \epsilon_j \in Y = \bigoplus_{j=1}^d\mathbb{Z} \epsilon_j$, the element $\varphi(\lambda)$ defined in~\eqref{def:varphis} is the following map:
\begin{equation}\label{eq:varphilambda}
\varphi(\lambda) :T_\kappa \to \mathbb{C}^\times,
\quad
\tau^a \mapsto  [\varpi^\lambda, \tau^a] = (u^l)^{B(\lambda, \tau^a)} = u^{\sum_j 2l\lambda_j a_j} = \zeta^{\sum_j 2\lambda_j a_j}.
\end{equation}
Therefore we may identify $\varphi(\lambda)$ with $\chi^{2l\lambda}$. 
\begin{lem}\label{lem:mpSkew}
The following relations holds in $\HGbIp$, for  
$\lambda \in Y$, 
$j, a_i \in[m]$:
\[
Y_{\epsilon_1} (t^j \otimes 1^{\otimes d-1}) =  \zeta^{2j} (t^j\otimes 1^{\otimes d-1}) Y_{\epsilon_1},
\qquad
Y_{\lambda} (t^{a_1} \otimes \dots \otimes t^{a_d}) =  \zeta^{2\sum_j a_j\lambda_j} (t^{a_1} \otimes \dots  \otimes t^{a_d}) Y_{\lambda}.
\]
\end{lem}\label{lem:YT}
\begin{proof}
Note that it follows from \eqref{eq:varphilambda} that
$\varphi(\lambda)_j(t) = u^{2l\lambda_j}$ for each $1\leq j \leq d$, and hence
\begin{equation}\label{cphi}
c(\varphi(\lambda)\chi^\gamma) \cong \frac{1}{m^d} 
\sum_{a \in[m]^d} 
(u^{2l\lambda_1+\gamma_1} t)^{a_1} \otimes \dots \otimes (u^{2l\lambda_d + \gamma_d} t)^{a_d}
=
c(\chi^{(\gamma_1+2l\lambda_1, \dots, \gamma_d+2l\lambda_d)}).
\end{equation}
It then follows from \eqref{tintot'} and \eqref{cphi} that
\begin{equation}\label{x_1t}
\begin{split}
Y_{\epsilon_1}(t\otimes 1^{\otimes d-1}) 
&=  Y_{\epsilon_1} \Big(
\big(\sum\nolimits_{i=1}^m u^{-i} (\tp{i})\big) \otimes 1^{\otimes d-1}\Big) 
=  \Big(
\big(\sum\nolimits_{i=1}^m u^{-i} (\tp{i+2l})\big)\otimes 1^{\otimes d-1}
\Big)  Y_{\epsilon_1}
\\
&=  u^{2l}\Big(
\big(\sum\nolimits_{i=1}^m u^{-i-2l} (\tp{i+2l})\big)\otimes 1^{\otimes d-1}
\Big) 
 Y_{\epsilon_1}
= \zeta^{2} (t\otimes 1^{\otimes d-1})  Y_{\epsilon_1}.
\end{split}
\end{equation}
The second formula then follows from applying \eqref{x_1t} iteratively.
\end{proof}

\subsection{Identification of the Hecke algebras}
\subsubsection{Pro-$p$ Iwahori Hecke algebra}
Recall  that the set 
$\{(t^{a_1} x^{\lambda_1} \otimes \dots \otimes t^{a_d} x^{\lambda_d} )H_w ~|~ a_i \in [m], \lambda_i \in \ZZ, w\in \Sigma_d\}$ is a basis of $\bbA \wr \cH(d)$ from  Theorem \ref{thm:SPQWPPBW}.
One can replace the toral elements by $\tp{j}$'s (see \eqref{tintot'}) to obtain the following basis:
\begin{equation}
\{((\tp{\gamma_1}) x^{\lambda_1} \otimes \dots \otimes (\tp{\gamma_d}) x^{\lambda_d} )H_w
~|~ a_i \in [m], \lambda_i \in \ZZ, w\in \Sigma_d\}.
\end{equation}
Define a $\CC$-linear map $\Psi:\bbA  \wr \cH(d) \to \HGbIp$ by
\begin{equation}
\Psi((
(\tp{\gamma_1}) x^{\lambda_1} \otimes \dots \otimes (\tp{\gamma_d}) x^{\lambda_d} )H_w) 
:=  
c(\chi^{\gamma}) Y_{\lambda} \cT_w,
\end{equation}
where $\lambda = \sum_{j=1}^m \lambda_j \epsilon_j \in Y$,
$\gamma = (\gamma_1, \dots, \gamma_d) \in [m]^d$,
$w\in \Sigma_d$.

\begin{thm}\label{thm:PQWP2}
The map $\Psi:\bbA  \wr \cH(d) \to \HGbIp$  is an algebra isomorphism.
\end{thm}
\begin{proof}
First, note that equation~\eqref{eq:MpPoly} ensures the well-definedness of the map $\Psi$ with respect to the definition of $\bbB^{\otimes d}$ in~\eqref{eq:def:FB}. 
	
Next, we show that $\Psi$ is an algebra homomorphism.
The braid relations (see \eqref{def:BR} and \eqref{eq:MpBR}) are preserved by $\Psi$.
Regarding the quadratic relations note that we have~\eqref{eq:MpQR}: 
\[
c_\sroot(\mathbf{1}) = \frac{1}{q-1}\sum_{u \in \kappa^*} \mathbf{1}(u) \cT_{\bar{h}_{\sroot}(u)} 
= \frac{1}{q-1} \sum_{u \in \kappa^*} \cT_{\bar{h}_{\sroot}(u)}
= \frac{1}{q-1} \sum_{1 \leq j \leq q-1} \cT_{\bar{h}_{\sroot}(t^j)}.
\]
In our identification, $\Psi(t_i^j t_{i+1}^{-j}) = \cT_{\bar{h}_{\sroot_i}(t^j)}$, and as such $\Psi(S_i) = (q-1) c_\sroot(\mathbf{1})$.
Moreover, we have that 
\[
\Psi(R_i) = q \mathcal{T}_{\bar{h}_{\alpha_i(-1)}},
\]
where $R$ is defined in~\eqref{def:R-Yokonuma}, as we may choose the standard embedding of $s_{\sroot_i}$ into $\mathbf{GL}_d(F)$ such that $s_{\sroot_i}^2=1$ (see~\cite[p. 715]{CS16} for $n=1$; nothing essential changes in the metaplectic setting $n>1$).
Then~\eqref{def:QR} allows us to conclude that $\Psi$ preserves the quadratic relations.

For the wreath relations \eqref{def:WR}, 
we only need to check the special case for $d=2$, i.e., 
\[
\Psi\big(
H_1 (\tau'_a \otimes \tau'_b) (x^i \otimes x^j) 
- (\tau'_b \otimes \tau'_a) (x^j\otimes x^i) H_1 
- (\tau'_b \otimes \tau'_a) \rho(x^i\otimes x^j)
\big)
=0. 
\]
Let us break this check into several smaller checks. 
The commutation between $H_1$ and $\tau'_a\otimes \tau'_b$  is given in~\eqref{eq:commutationHf} on the algebraic side and left side of ~\eqref{eq:MpSkew} on the $p$-adic side; they are the same. 
The commutation between $x^i\otimes x^j$ and $\tau'_a\otimes \tau'_b$ is given by Lemma~\ref{lem:mpSkew} on the algebraic side and right side of~\eqref{eq:MpSkew} on the $p$-adic side; these match as well noting the explicit formula for $\varphi(\lambda)$ in \eqref{eq:varphilambda}. 
The commutation between $H_1$ and $x^i\otimes x^j$ is explicitly written in~\eqref{eq:commutationHP}. 
The corresponding $p$-adic relation is ~\eqref{eq:MpWR} where we note that our extra assumption $q-1\equiv0\pmod{2n}$ implies $\epsilon=1$ in~\eqref{eq:MpWR} and that the definition of $c_{\sroot_i}(j)$ in~\eqref{eq:def:calphai} and Lemma~\ref{lem:SkewDemazureFacts} implies the equivalence between the algebraic and $p$-adic equations.
This finishes the proof that $\Psi$ is a homomorphism.

Finally, $\Psi$ is an isomorphism by Theorem \ref{thm:SPQWPPBW} and Corollary~\ref{cor:propHbasis}.
\end{proof}

\begin{rmk}
	The $n=1$ version of Theorem~\ref{thm:PQWP2} is proved in~\cite[Thm. 4.1]{CS16}.
\end{rmk}

The isomorphism in Theorem~\ref{thm:PQWP2} induces an isomorphism of subalgebras 
\begin{equation}
	\Fr \wr \cH(d) \cong \HGbIpk.
\end{equation}
Recall the metaplectic Iwahori Hecke algebra $\HGbI$ and the algebra $\CC[x^{\pm \bar{n}}]\wr\cH(d)$ of Example~\ref{ex:maha}.
One may similarly prove 
\begin{prop}\label{prop:IHtoPQWP}
	The map $\Psi_I:\CC[x^{\pm \bar{n}}]\wr\cH(d) \to \HGbI$ sending 
    $(x^{\bar{n}\lambda_1}\otimes \cdots \otimes x^{\bar{n}\lambda_d}) H_w$ to $Y_{\bar n\lambda} \cT_w$ for all $\lambda \in Y$ is an algebra isomorphism.
\end{prop}
We skip the proof of this Proposition, but note that it follows from a presentation of $\HGbI$ (given in~\cite{S88,S04,McN12}, see also~\cite[\S4.6]{GGK1}) and proceeds along the lines of the previous proof.

\subsubsection{Iwahori Hecke algebra}
From now on, write $t' = \tp{0}$ for short.
Let $\epsilon_I := t' \otimes t' \otimes \dots \otimes t'\in \Fr^{\otimes d}$. 
Then $\epsilon_I$ is the image under $\Psi^{-1}$ of $e_I$ defined just before~\eqref{eq:proptoIwahori}.
We will now show how to derive ~\eqref{eq:proptoIwahori} algebraically at the level of quantum wreath products. 
Let us start with a computational lemma.
\begin{lem}\label{lem:idemlem}
The following equalities hold in $\bbA \wr \cH(d)$:
\begin{enua}
\item $(t')^2 = t'$, and hence $\epsilon_I^2 = \epsilon_I$.
\item $\epsilon_I Y_\lambda \epsilon_I = 0$ unless $\lambda \in Y_{(Q,n)} = \bar{n} Y$.
\item $t' t = tt' = t'$, and hence $\epsilon_I f = f \epsilon_I$ for all $f \in \Fr^{\otimes d}$.
\item $H_i \epsilon_I = \epsilon_I H_i$ for all $1\leq i<d$.
\item $\epsilon_I t_i = t_i \epsilon_I  = \epsilon_I$ for all $1\leq i \leq d$.
In particular, $e_i \epsilon_I = \epsilon_I = \epsilon_I e_i$.
\end{enua}
\end{lem}
\proof
Parts (a), (c) and (d) follow from a direct verification.
For part (b), pick any $i$, by (a) we have $\epsilon_I = t'_i \epsilon_I = \epsilon_I t'_i $, and hence
\[
e_I Y_\lambda e_I =  e_I Y_\lambda t'_i e_I 
= e_I  \zeta^{2\lambda_i}  t'_i Y_\lambda  e_I 
=  \zeta^{2\lambda_i}  e_I Y_\lambda  e_I .
\]
Thus, $0 = (1-\zeta^{2\lambda_i }) e_I Y_\lambda  e_I$. 
That is, $e_I Y_\lambda  e_I  = 0$  unless $\lambda_i$ divides $\bar{n}$ for any $i$, which is equivalent to that $\lambda \in Y_{Q,n}$.
Finally, part (e) follows from (c) and the fact that
\[
\begin{split}
(t'\otimes t')e 
&= \frac{1}{m}\sum\nolimits_{j=1}^m (t't^j)\otimes (t't^{-j}) 
= \frac{1}{m} \sum\nolimits_{j=1}^m t'\otimes t'
\\
&= (t'\otimes t')
\\
&=  \frac{1}{m}\sum\nolimits_{j=1}^m (t^jt')\otimes (t^{-j}t')  = (t'\otimes t')e.\qedhere
\end{split} 
\]
\endproof

\begin{prop}\label{thm:eHe}
The following map is an algebra isomorphism:
\[
\Upsilon: \CC[x^{\pm \bar{n}}]\wr\cH(d) \to \epsilon_I (\bbA\wr\cH(d)) \epsilon_I,
\quad
x^{\bar{n}\lambda}
\mapsto \epsilon_I x^{\bar{n}\lambda}\epsilon_I.
\]
\end{prop}
\proof
Thanks to \cref{thm:SPQWPPBW} and \cref{lem:idemlem}(b), this map $\Upsilon$ is a bijection, and it remains to check that $\Upsilon$ is also an algebra homomorphism.
Note that by \cref{lem:idemlem}(a) and (d),  
\[
\Upsilon(H_iH_j)=
(\epsilon_I H_i \epsilon_I)(\epsilon_I H_j\epsilon_I)
=
\epsilon_I H_i \epsilon_I^2 H_j \epsilon_I
=
\epsilon_I H_i \epsilon_I H_j \epsilon_I
= 
\epsilon_I H_i  H_j \epsilon_I^2
=
\epsilon_I H_i H_j \epsilon_I
=
\Upsilon (H_iH_j),
\]
and hence the braid relations are preserved.
For the quadratic relations, we have
\[
\Upsilon(H_i)^2
=
\epsilon_I H_i  \epsilon_I H_i \epsilon_I
=
\epsilon_I H_i^2  \epsilon_I
=
\epsilon_I (S_iH_i+R_i)  \epsilon_I
=
\epsilon_I S_i H_i\epsilon_I^2 +  q \epsilon_I,
\]
where we use \cref{lem:idemlem}(a) and (d). 
Note that by \cref{lem:idemlem}(e) we have $ \epsilon_I S_i= (q-1) \epsilon_I$.
Thus, $\Upsilon(H_i)^2 = (q-1) \Upsilon(H_i) + q\epsilon_I = \Upsilon(H_i^2)$.

Finally, for the wreath relations, it suffices to verify the $d=2$ case when $\lambda$ is the minuscule weight, i.e., $\lambda = \bar{n} \epsilon_1$. 
It follows from \cref{lem:SkewDemazureFacts}(b) that, by setting $P = x_1^{\bar{n}}$, the following holds in $\bbA\wr\cH(d)$:
\[
H_1 x_1^{\bar{n}} 
=  x_2^{\bar{n}}  H_1 + S x_1^{\bar{n}} + x_2 S x_1^{\bar{n}-1}
+ \dots + x_2^{\bar{n}-1} S x_1.
\]
Applying \cref{lem:idemlem}(b) and (e), we obtain
\[
\begin{split}
\Upsilon(H_1) \Upsilon(x_1^{\bar{n}})
&=(\epsilon_I H_1 \epsilon_I)(\epsilon_I x_1^{\bar{n}} \epsilon_I)
\\
&= \epsilon_I( x_2^{\bar{n}} H_1 + S x_1^{\bar{n}} + x_2 S x_1^{\bar{n}-1}
+ \dots + x_2^{\bar{n}-1} S x_1)\epsilon_I
=
(\epsilon_I x_2^{\bar{n}}\epsilon_I)(\epsilon_I  H_1 \epsilon_I) + (\epsilon_I Sx_1^{\bar{n}}\epsilon_I)
\\
&= \Upsilon(x_2^{\bar{n}}) \Upsilon(H_1) + (q-1) \Upsilon(x_1^{\bar{n}}).
\end{split}
\]
Therefore the map $\Upsilon$ is indeed an isomorphism.
\endproof
\cref{thm:eHe} is an algebraic version of the isomorphism~\eqref{eq:proptoIwahori}. 
It shows how skew quantum wreath products naturally keep track of reduction in support in $\HGbI$ (\cref{lem:idemlem}(b)), as well as imply 
\begin{equation}\label{eq:Iwahoriqwp}
	\HGbI \cong \CC[x^{\pm \bar{n}}]\wr\cH(d). 
\end{equation}

\subsection{Identification of Hecke modules}
We want to compare the module $\WbIp$ introduced in~\eqref{eq:defWIp} with $\bbV(1) \wr \sgn$ as modules over $\HGbIp \cong \bbA  \wr \cH(d)$. 
Before we can do that, we need to rescale the basis of $\WbIp$ in order to ``get rid'' of the Gauss sums. 
These Gauss sums do not appear naturally on the algebraic side, but contain important number theoretic information, so we want to normalize them away in such a way that we can reintroduce them back later if necessary. 

\subsubsection{Rescaling the $\Vk$ and $\WbIp$ bases}
Recall the basis $\{c^\bg(\chi)\}$ of $\Vk$ from \S\ref{subsub:finiteGG}. 
Identify $\chi \in \Hom(T_\kappa, \CC^*)$ with $\gamma \in [m]^d$ as we did in \S\ref{subsec:toral}. 
The action of $\cT_{s_i}\in \HGbIpk$ from~\eqref{eq:action-finite-Vk} reads in this setup:
\begin{equation}\label{eq:action-finite-Vk-2}
	c^\bg(\chi) \cT_{\sroot_i} = \bg_{\gamma_i-\gamma_{i+1}} c^\bg(s_i \chi).
\end{equation}
Let us define a new basis $c^\bq(\chi)$ of $\Vk$ as follows:
\begin{equation}
	c^\bq (\chi) := \left( \prod_{\gamma_i>\gamma_{i+1}} \frac{\sqrt{q}}{\bg_{\gamma_i-\gamma_{i+1}}} \right) c^\bg(\chi).
\end{equation}
\begin{lem}
	The action in~\eqref{eq:action-finite-Vk-2} on the basis $\{c^\bq (\chi)\}$ reads:
	\begin{equation}\label{eq:tiactionfinitepadic}
		c^\bg(\chi) \cT_{\sroot_i} =
		\begin{cases}
			-c^\bg(s_i \chi), &\textup{if } \gamma_i=\gamma_{i+1};
			\\
			 \sqrt{q}c^\bg(s_i \chi), &\textup{if } \gamma_i \neq \gamma_{i+1}.
		\end{cases}
	\end{equation}
\end{lem}
\begin{proof}
	If $\gamma_i=\gamma_{i+1}$, then $c^{\bg / \bq}(\chi) = c^{\bg / \bq}(s_i\chi)$ and $\bg_0=-1$, so $c^\bg(\chi) \cT_{\sroot_i} = - c^\bg(s_i \chi) \iff c^\bq(\chi) \cT_{\sroot_i} = - c^\bq(s_i \chi)$. 
	Otherwise, write $c^\bq(\chi) = x c^{\bg}(\chi)$.
	
	If $\gamma_i>\gamma_{i+1}$, then 
	$\frac{\sqrt{q}}{\bg_{\gamma_i-\gamma_{i+1}}} c^\bq(s_i \chi) = x c^{\bg}(s_i \chi)$. 
	So~\eqref{eq:action-finite-Vk-2} implies $ c^\bq(\chi) \cT_{\sroot_i} = \sqrt{q} c^\bq(s_i \chi)$.
	
	If $\gamma_i<\gamma_{i+1}$, then $\frac{\bg_{\gamma_{i+1}-\gamma_{i}}}{\sqrt{q}} c^\bq(s_i \chi) = x c^{\bg}(s_i \chi)$ . 
	By~\eqref{eq:gauss-sums-properties}, we have $\frac{\bg_{\gamma_{i+1}-\gamma_{i}} \bg_{\gamma_{i}-\gamma_{i+1}}}{\sqrt{q}} = \sqrt{q}$, and then~\eqref{eq:action-finite-Vk-2} implies $ c^\bq(\chi) \cT_{\sroot_i} = \sqrt{q} c^\bq(s_i \chi)$ again.
\end{proof}

The new basis ${c^\bq(\chi)}$ of $\Vk$ induces a new basis $\{c^\bq(\chi) Y_\lambda\}$ of $\WbIp$.
Because of the description of $\WbIp \cong \Vk \otimes_{\HGbIpk} \HGbIp$ in Proposition~\ref{prop:affineWIpinduced}, we have that equation~\eqref{eq:tiactionfinitepadic} holds in the $\WbIp$ as well. 
\subsubsection{The main results}

We work with $N=1$, that is $\bbV(1) = M_1$.
Combining \eqref{eq:MidbasisA}, Proposition \ref{prop:MlambdaA}, and \eqref{tintot'}, the set
$\{v'_{a,\lambda} ~|~ a= (a_1, \dots, a_d)\in[m]^d, \lambda \in \ZZ^d\}$ forms a $\CC$-basis of the
$\bbA\wr\cH(d)$-module $\bbV(1)^{\otimes d} \cong \bbV(1)\wr \sgn$, where
\begin{equation}
v'_{a,\lambda} := 
v^+
(\tp{a_1}\otimes \dots \otimes \tp{a_d})
(x^{\lambda_1}\otimes \dots \otimes x^{\lambda_d}),
\quad
v^+ := v_1^{\otimes d}.
\end{equation}    

\begin{thm}\label{thm:VGG2}
There is an $\bbA \wr \cH(d) \cong \HGbIp$-module isomorphism 
\[
\bbV(1)^{\otimes d} \cong 
\bbV(1) \wr \sgn  
\cong \WbIp
\]
determined by $v'_{a,0} \mapsto c^\bq(\chi^{a})$ for $a_i \in [m]$.
\end{thm}
\begin{proof}
It suffices to consider the $d=2$ case. 
First assume $a= (i,j)$ and $\lambda = 0$. Then, $v'_{a,0} = v^+ (\tp{i}\otimes \tp{j})$, and hence the right $H_1$-action on $v'_{a,0}$ is given by
\begin{align}
\label{eq:v'a0H1}
v'_{a,0} \cdot H_1 &=   v^+ (\tp{i}\otimes \tp{j}) \cdot H_1
\\
\notag&= v^+ \cdot H_1 (\tp{j}\otimes \tp{i}) &\textup{by }\eqref{Leib0}
\\
\notag&= -v^+ \gamma (\tp{j}\otimes \tp{i}). &\textup{by } \eqref{eq:fundaction2}
\end{align}
Using Lemma \ref{lem:idemlem}, we obtain that
    $e(\tp{i}\otimes\tp{j}) = 
        \delta_{i,j} (\tp{i}\otimes\tp{j})$, 
and hence, for $i\neq j$:
\begin{align}
    &\gamma (\tp{j}\otimes \tp{i})
    = (\sqrt{q} + 1)e(\tp{j}\otimes \tp{i}) - \sqrt{q}(\tp{j}\otimes \tp{i})
    =- \sqrt{q}(\tp{j}\otimes \tp{i}),
    \\
    &\gamma (\tp{i}\otimes \tp{i})
    = (\sqrt{q} + 1)e(\tp{i}\otimes \tp{i}) - \sqrt{q}(\tp{i}\otimes \tp{i})
    = (\tp{i}\otimes \tp{i}).
\end{align}
Therefore, \eqref{eq:v'a0H1} becomes 
\begin{equation}
    v'_{(i,j),0} \cdot H_1 = \begin{cases}
        - v'_{(i,j), 0}, &\textup{if } i=j;
        \\
        \sqrt{q} v'_{(j,i), 0}, &\textup{if } i\neq j.
    \end{cases}
\end{equation}
This action matches with the action of $\cT_{\sroot_1}$ on $c^\bq(\chi)$ with corresponding $\gamma = (\gamma_1,\gamma_2)=(i,j)$ from equation \eqref{eq:tiactionfinitepadic}. 

Secondly, let $a = (a_i)_i, b = (b_i)_i\in [m]^2$ and $\lambda = (\lambda_i)_i, \mu = (\mu_i)_i \in \ZZ^2$. Then we have that
\begin{align}
  &v'_{a,\lambda} \cdot (x^{\mu_1}\otimes x^{\mu_2}) 
  = v'_{a, \lambda+\mu},
  \\
  &v'_{a,\lambda} \cdot (\tp{b_1}\otimes \tp{b_2}) 
  = v'_{a,0} (\tp{b_1+l\lambda_1}\otimes \tp{b_2+l\lambda_1})(x^{\lambda_1}\otimes x^{\lambda_2}) ,
  \\
  &v'_{a, \lambda} \cdot H_1
     = 
     - v'_{a, 0} \rho(x^{\lambda_2}\otimes x^{\lambda_1})
     +\begin{cases}
     v'_{a, (\lambda_2,\lambda_1)} (\sqrt{q})            &\textup{if } a_1=a_2;
     \\ \label{eq:actionH1}
     v'_{a, (\lambda_2,\lambda_1)} (-1)
            &\textup{if } a_1\neq a_2,
     \end{cases}
\end{align}
where the second line uses the fact that $x(\tp{i}) = (\tp{i+l})x$, 
and the third line can be further deduced using Lemma \ref{lem:SkewDemazureFacts}.
The matching between the action of $\cT_{\sroot_i}$ on $c^\bq(\chi_{a}) Y_{\lambda}$ and the action in~\eqref{eq:actionH1} follows from the isomorphism in Theorem~\ref{thm:PQWP2} as well as the description of $\WbIp$ from Proposition~\ref{prop:affineWIpinduced}.
This finishes the proof.
\end{proof}

Recall the $p$-adic Schur algebra $\SGbIp:= \End_{\HGbIp} (\WbIp)$ from~\eqref{eq:def:padicSchur}. 
Then Theorem~\ref{thm:VGG2} implies 

\begin{cor}\label{cor:Schuriden}
	There is an algebra isomorphism $\SGbIp \cong \bbS(1,d)$.
\end{cor}
Consequently, Proposition~\ref{prop:mult2} now gives a basis with a multiplication formula for the algebra $\SGbIp$. 

For $N> 1$, it will be interesting to study the corresponding Schur algebras $\bbS(N,d)$, 
their associated ``quantum groups'',
and possible connections to $p$-adic representation theory.

\subsection{Schur algebras at the Iwahori level}
Denote the original Kashiwara--Miwa--Stern tensor space in~\cite[(32)]{KMS95} by 
$V(\eta)^{\otimes d}\in\mod\hp\CC[x^{\pm\bar{n}}]\wr\cH(d)$; it is a module over the affine Hecke algebra depending on an integer $\eta\geq 1$.
Here,  $V(\eta)= \bigoplus_{i\in\ZZ} \CC \mrm{v}_i$ is the $\CC[x^{\pm\bar{n}}]$-module such that $x^{\pm\bar{n}}$ acts by $\mrm{v}_i \cdot x^{\pm\bar{n}} = \mrm{v}_{i\pm \eta}$.
Write
$\mrm{v}_f:= \mrm{v}_{f(1)} \otimes \cdots \otimes \mrm{v}_{f(d)}$ for any $f:[d] \to \bbZ$.
We note that this space can be recovered using the formulas in Section \ref{sec:KMSforSPQWP} by setting:
\[
m=1,
\quad
n=1,
\quad
\Fr = \CC,
\quad
e = 1,
\quad
\gamma = 1.
\]
Therefore, each $x_i^{\bar{n}}$ acts by   
\[
\mrm{v}_f \cdot x_i^{\bar{n}} = 
\mrm{v}_{f(1)} \otimes \dots \otimes \mrm{v}_{f(i-1)}
\dots \otimes \mrm{v}_{f(i)+\eta} \otimes \mrm{v}_{f(i+1)} \otimes \dots \otimes \mrm{v}_{f(d)},
\quad f:[d]\to \ZZ.
\] 

Recall the space $\bbV(N)^{\otimes d}$ introduced in \S\ref{sec:KMSforSPQWP} for $N \geq 1$. 
In light of the corner algebra identification (Proposition \ref{thm:eHe}), 
the space $\bbV(N)^{\otimes d}\epsilon_I$ is a module over 
$\epsilon_I (\bbA\wr\cH(d)) \epsilon_I$, i.e., the affine Hecke algebra.
\begin{prop}\label{prop:KMSIwa}
As modules over $\CC[x^{\pm\bar{n}}]\wr\cH(d) \cong \epsilon_I (\bbA\wr\cH(d)) \epsilon_I$, there is an isomorphism
\[
\Xi: V(N\bar{n})^{\otimes d} \to \bbV(N)^{\otimes d}\epsilon_I,
\quad
\mrm{v}_f  \mapsto v_f \epsilon_I.
\]
\end{prop}
\begin{proof}
Firstly, 
recall that $t'$ commutes with $x^{\bar{n}}$ in $\bbA$ since
\[
x^{\bar{n}} t' = x^{\bar{n}} (\tp{0})
= (\tp{2l\bar{n}}) x^{\bar{n}}
= (\tp{0}) x^{\bar{n}},
\]
where the second equality uses \cref{lem:mpSkew}, and the third equality uses the fact that $n$ divides $2\bar{n}$.
One then deduces from Lemma \ref{lem:idemlem}(a) that  
$(v_i t') \cdot (t' x^{\bar{n}} t') = v_{i+N\bar{n}} t' \in \bbV(N) t'$.
Recall the isomorphism $\Upsilon$ from \cref{thm:eHe}.
We have
\[
(v_f \epsilon_I) \cdot \Upsilon(x_i^{\bar{n}})
=
(v_f \epsilon_I) \cdot (\epsilon_I x_i^{\bar{n}} \epsilon_I) 
= v_{f} x_i^{N\bar{n}} \epsilon_I 
= \Xi(\mrm{v}_f \cdot x_i^{\bar{n}}).
\]
On the other hand, since $e\epsilon_I = \epsilon_I$, we have
\begin{equation}\label{eq:degKMSaction}
\gamma \epsilon_I = (\sqrt{q} + 1)e \epsilon_I - \sqrt{q} \epsilon_I = \epsilon_I,
\quad
S \epsilon_I = (q-1)\epsilon_I,
\end{equation}
and hence
\[
(v_f \epsilon_I) \cdot \Upsilon(H_i)
=
(v_f \epsilon_I) \cdot (\epsilon_I H_i\epsilon_I) 
= \Xi(\mrm{v}_f \cdot H_i),
\]
where the last equality follows from a case-by-case comparison by combining \cref{prop:KMSaction} and \eqref{eq:degKMSaction}. 
We have verified that $\bbV(N)^{\otimes d}\epsilon_I \cong V(N\bar{n})^{\otimes d}$ as modules over the affine Hecke algebra. 
\end{proof}
Therefore, for any $\bar{n} \geq 1$, the KMS tensor space $V(\bar{n})^{\otimes d}$ can be recovered from the pro-$p$ module $\bbV(1)^{\otimes d}$ via the idempotent $\epsilon_I$.
Moreover, one has the following $p$-adic corollary. 

\newcommand{\Haff}{Haff}

From~\eqref{eq:proptoIwhaoriGG}, \eqref{eq:Iwahoriqwp} and the identification of $e_I$ with $\epsilon_I$ and~\cref{prop:KMSIwa} we get that
\begin{equation}\label{eq:KMStoWI}
	\WbI \cong V(\bar{n})^{\otimes d} \textup{ as } \HGbI \cong \CC[x^{\pm\bar{n}}]\wr\cH(d) \textup{ modules.}
\end{equation}
Denote the well known quantum affine Schur algebra by $\widehat{S}_{\sqrt{q}}(\bar{n},d) := \End_{\CC[x^{\pm\bar{n}}]\wr\cH(d)}(V(\bar{n})^{\otimes d})$~\cite{G99}. 
It is known to be a quotient of the quantum affine group $U_{\sqrt{q}}(\widehat{\mathfrak{gl}}_{\bar{n}})$. 
\begin{cor}
	We have the following isomorphism of Schur algebras:
	\begin{equation}
		\End_{\HGbI} \WbI \cong \End_{\CC[x^{\pm\bar{n}}]\wr\cH(d)} (V(\bar{n})^{\otimes d}) \cong	  \widehat{S}_{\sqrt{q}}(\bar{n},d) \twoheadleftarrow U_{\sqrt{q}}(\widehat{\mathfrak{gl}}_{\bar{n}}).
	\end{equation}
\end{cor}
This Corollary and equation~\eqref{eq:KMStoWI} recover the main results of~\cite[\S4]{GGK2}.

\bibliography{bl}{}
\bibliographystyle{alphaabbr}

\end{document}